\documentclass{amsart}
\usepackage{CZ}
\usepackage{amsaddr}
\fullpage

\title{Hilbert's third problem and a conjecture of Goncharov}
\author{Jonathan A. Campbell}
\address{Center for Communications Research, La Jolla}
\author{Inna Zakharevich}
\address{Cornell University} 
\begin{document}

\maketitle
\begin{abstract}
  In this paper we reduce the generalized Hilbert's third problem about Dehn
  invariants and scissors congruence classes to the injectivity of certain
  Cheeger--Chern--Simons invariants. We also establish a version of a conjecture of
  Goncharov relating scissors congruence groups of polytopes and the algebraic
  $K$-theory of $\mathbf{C}$. We prove, in particular, that the homology of the
  ``Dehn complex'' of Goncharov splits as a summand of the twisted homology of a
  Lie group made discrete.
\end{abstract}


\section*{Introduction}

Hilbert's third problem asks the following question: given two polyhedra $P$ and
$Q$, when is it possible to decompose $P$ into finitely many polyhedra and form
$Q$ out of the pieces?  More formally, is it possible to write
$P = \bigcup_{i=1}^n P_i$ and $Q = \bigcup_{i=1}^n Q_i$ such that $P_i\cong Q_i$
for all $i$, and such that $\meas(P_i\cap P_j) = \meas(Q_i\cap Q_j) = 0$ for all
$i \neq j$?  (Two polyhedra for this this is true are called \textsl{scissors
  congruent}.)  The generalized version of Hilbert's third problem is the
observation that this can be asked in any dimension and any geometry.  The
question then becomes: describe a complete set of invariants of scissors
congruence classes of polytopes in a given dimension and geometry.

Let us briefly consider the classical version.  If two polyhedra are scissors
congruent then their volumes are equal.  The reverse implication is not true; a
second invariant, called the \textsl{Dehn invariant}, exists.  For
three-dimensional polyhedra (in Euclidean, spherical or hyperbolic space) this
invariant is defined as follows:
\[D(P) = \sum_{\substack{e\mathrm{\ edge\ of\ }P}} \mathrm{len}(e) \otimes
  \theta(e) \in \R\otimes \R/\pi\Z.\] Here, $\theta(e)$ is the dihedral angle at
$e$; in other words, it is the arc length of the intersection with $P$ of a
small circle around $e$.  In dimension $n$ it is possible to define other Dehn
invariants, by picking a dimension $\ell$ and writing a similar sum over all
faces of $P$ of dimension $\ell$; the measure of the angle will then be a
portion of the sphere in dimension $n-\ell-1$.  By the Dehn--Sydler theorem
\cite{sydler65, jessen68} in Euclidean space in dimensions $3$ and $4$, two
polytopes are scissors congruent if and only if their volumes and Dehn
invariants are equal.  Work of Dupont and Sah \cite{dupont_sah_ii} extended this
technique to$3$-dimensional spherical and hyperbolic space, classifying the
kernel of the Dehn invariant as a group homology group.  We thus have the
following question:

\begin{mainconjecture}[{\cite[Question 1]{dupont_sah_ii}}]
  In Euclidean, spherical, and hyperbolic geometries, do the volume and generalized
  Dehn invariant separate the scissors congruence classes of polytopes?
\end{mainconjecture}

\begin{rmknonum}
  Spherical polytopes are often used to measure angles.  When a polytope is
  decomposed into smaller polytopes it produces extra angles, at all of the
  faces along the cuts.  These newly produced angles always add up either to an
  entire sphere (if they are contained in the interior of the original polytope)
  or to some type of ``flat'' angle.  Such flatness can be quantified by
  observing that such angles always arise from an angle in a lower-dimensional
  sphere; in spherical scissors congruence classes we thus declare all such
  angles to have ``scissors congruence measure $0$.''  The Dehn invariant, and
  thus the generalized version of Hilbert's third problem, can also be defined
  for polytopes up to such ``measure-$0$'' polytopes.  See
  Definition~\ref{def:spherangle}. 
\end{rmknonum}

An algebraic approach to Hilbert's Third Problem defines \emph{scissors
  congruence groups}, which are free abelian groups generated by polytopes (in
whichever geometry is under consideration) modulo ``cutting'' and translation by
isometries.  (For a more formal definition, see Definition~\ref{def:scissors}.)
Both volume and the generalized Dehn invariant can then be defined as
homomorphisms of groups, and we see that the generalized Hilbert's Third Problem
has a positive answer exactly when volume is injective when restricted to the
kernel of the Dehn invariant.

Motivated by the theory of mixed Tate motives, in \cite[Conjecture
1.7]{goncharov} Goncharov proposed the following method for solving generalized
Hilbert's third problem.  Let
$D:\P(S^{2n-1}) \to \bigoplus_{i=1}^n \P(S^{2i-1})\otimes \P(S^{2(n-i)-1})$ be
the Dehn invariant on the (reduced) spherical scissors congruence groups.  The
generalized Hilbert's third problem can then be rephrased to say that volume is
injective when restricted to $\ker D$.  Conjectures of Ramakrishnan
\cite[Conjectures 7.1.2,7.1.8]{ramakrishnan} imply that the Borel regulator
produces an injective homomorphism
\[(\gr_n^\gamma K_{2n-1}(\CC)_\Q\otimes (\Q^\sigma)^{\otimes n})^+ \to
  \R/(2\pi)^n\Q.\] Here, $\Q^\sigma$ is $\Q$ with an action of $\Z/2$ via the
sign, $\gr_n^\gamma K_{2n-1}(\CC)$ is the $n$-th graded piece of the weight
filtration on the algebraic $K$-theory of $\CC$ (with action by $\Z/2$ induced
by complex conjugation), and $\cdot^+$ denotes taking the fixed points of the
action.  (For a more detailed explanation of this, see Section~\ref{sec:conj};
an in-depth understanding of the terms is not needed for the current
discussion.)  If it were possible to construct an injective map
\begin{equation} \label{eq:gonchmap}
  \ker D \otimes \Q \rcofib (\gr_n^\gamma K_{2n-1}(\CC)_\Q\otimes (\Q^\sigma)^{\otimes
    n})^+
\end{equation} 
such that the composition with the Borel regulator was equal to the volume, this
would imply generalized Hilbert's third problem for spherical scissors
congruence groups (at least modulo torsion).  Goncharov also made an analogous
form of this conjecture for the hyperbolic groups; here the Borel regulator
takes values in $\R$.  In his paper, Goncharov was able to construct a map of
the form (\ref{eq:gonchmap}) once polytopes were restricted to polytopes with
algebraic vertices and $\CC$ was replaced by $\bar \Q$; however, he did not
show that it is injective.

There are a couple of indications this initial version may not be the most
useful form of the conjecture.  Restricting to polytopes of a particular
dimension restricts us to considering group homology for matrices of a set
dimension; this is directly related to the \textsl{rank} filtration, rather than
the 
$\gamma$-filtration.  As the graded pieces of the rank and $\gamma$-filtrations
are expected to be isomorphic (see, for example, \cite[Conjecture
2.6.1]{knudson_homology} for an in-depth discussion) this is not a major change
to the conjecture.  Since $K_*(\CC)_\Q$ is isomorphic to the primitive elements
in the Hopf algebra $H_*(BGL(\CC);\Q)$ (where $GL(\CC)$ is considered as a
discrete group), the desired map (\ref{eq:gonchmap}) can be described as a map
into a quotient of certain group homology groups.  The second is the observation
that scissors congruence groups are constructed out of the group homology of
orthogonal groups, rather than general linear groups, so it is more likely that
the kernel of the Dehn invariant will be related to the group homology of
orthogonal groups, rather than general linear groups.  It turns out that the
correct analog of the quotient in the orthogonal case is simply the groups
$H_*(O(n;\R);\Zinv^\tw)$, where $\cdot^\tw$ indicates that the group is acting
via multiplication by the determinant.  These also have a regulator, usually
referred to as the Cheeger--Chern--Simons class, which agrees with the Beilinson
(and thus Borel) regulator \cite{dupont_hain_zucker} and which is also expected
to be injective.

With these two changes we can prove the reverse of Goncharov's conjecture:
\begin{maintheorem}[Theorem~\ref{thm:hyperbolicGonch1.7}] \label{mainthm:a}
  Let $D$ be the Dehn invariant for hyperbolic scissors congruence.  There is a
  homomorphism
  \[H_d(O(1,d;\R);\Zinv^\sigma) \rto \ker D\]
  which, after composition with volume is equal to the Cheeger--Chern--Simons
  class.  Thus, if this map is surjective and the Cheeger--Chern--Simons class
  is injective, volume and the Dehn invariant separate scissors congruence
  classes in spherical geometry in all dimensions. 
\end{maintheorem}

An analogous statement is true for the spherical case, although it is somewhat
more complicated as volume is not well-defined on the reduced spherical scissors
congruence groups.  For more details, see Theorem~\ref{thm:sphericalGonch1.7}.

\begin{rmknonum}
  For unreduced spherical scissors congruence groups, it is already known that
  volume and Dehn invariant separate scissors congruence classes
  \cite[Proposition 6.3.22]{sah79}.  However, unreduced spherical scissors
  congruence groups do not appear to have nearly as many interesting
  applications as the reduced version, and we therefore focus on the reduced
  case.  In the reduced case, the even-dimensional reduced scissors congruence
  groups are known to be $0$ (see Proposition~\ref{prop:evencontr}), as is the
  homology group in the codomain, so the theorem is vacuous in these cases;
  however, we state it in full generality to make the analogy with the
  hyperbolic case clear.
\end{rmknonum}

Goncharov's intuition about scissors congruence classes did not stop at the
kernel of the Dehn invariant.  He noticed that Dehn invariants can be iterated
to produce a chain complex, denoted $\P_*(S^d)$.  (See Section~\ref{sec:goncharov} for more details.)  In \cite{goncharov}, he
conjectures \cite[Conjecture 1.8]{goncharov} that there exists a homomorphism
\[H_m(\P_*(S^{2n-1})\otimes \Q) \rto (\gr^\gamma_n K_{n+m}(\CC)_\Q\otimes
  (\Q^\sigma)^n)^+\]
for all $m$.  The techniques for proving Theorem~\ref{mainthm:a} extend to proving a
form of this conjecture, as well:
\begin{maintheorem}[Theorem~\ref{thm:projbase}] \label{mainthm:inj}
  Let $X = S^d$ or $\cH^d$, and let $I(X)$ be the isometry group of $X$.  
  For all $m$ there are homomorphims
  \[H_{m+\lfloor \frac{d-1}{2}\rfloor}(I(X),\Zinv^\sigma) \rto H_m(\P_*(X)).\]
\end{maintheorem}
In fact, Theorem~\ref{thm:projbase} is shown for any field of characteristic
$0$, not just $\R$; for the definition of scissors congruence groups over a
general field see Definition~\ref{def:gen-scissors}.

The main tool allowing us to prove these theorems is the ``geometrization'' of
the Dehn invariant: a topological model which is both \emph{rigid} and
\emph{equivariant with respect to the isometry group of our geometry}.  It is
rigid in the sense that the structural properties that we desire of the Dehn
invariant (described at the beginning of Section~\ref{sec:dehn}) already hold
for the topological spaces, without having to work ``up to homotopy'' or ``inside
homology groups.''  It is equivariant in the sense that the Dehn invariant is a
map of $I(X)$-spaces, rather than simply topological spaces.

The advantage of this construction is that the presence of higher homological
information in the coinvariant computations leads to major cancellations.  All
of the complexity of $\P_*(X)$ is contracted into $\Q^\tw$.  Here, the key
observation is that in a topological context homotopy coinvariants and the
``total complex'' that Goncharov uses to define $\P_\ast$ commute past one
another; thus the rigid and equivariant construction of the Dehn invariant above
can be used to explicitly determine the homotopy type of a space modeling this
complex.  We produce a spectral sequence whose lowest nonzero row is the complex
$\P_\ast(S^{2n-1})$ (resp. $\P_\ast(\cH^{2n-1})$); the cancellations allowing us
to identify the homotopy type of the ``total complex'' allows us to directly
relate this to the homology of $O(2n)$.

At the end of our analysis we illustrate the connection between our
reformulation and Goncharov's original conjectures (see
Proposition~\ref{cor:zigzags}). 

\subsection*{Outline of the Proof of Theorem~\ref{mainthm:inj}} The key
ingredient in this proof is the repeated use of the notions of homotopy cofiber
and homotopy coinvariants.

The Dehn complex is defined as the total complex of a cubical diagram in $\AbGp$
(Definition~\ref{def:dehncomplex}), each vertex of which is obtained by taking
coinvariants (i.e. $H_0$) of an action on a Steinberg module. The total complex
of a cube is the the same as the total homotopy cofiber taken in the category of
chain complexes (Example~\ref{example:cube}). Thus, to construct the Dehn
complex, one takes homology of a group, and then takes a homotopy colimit. This
order feels unnatural from the point of view of homotopy theory, as one should
first construct the space and then analyze its homology.

We begin by replacing each Steinberg module with a space
(Definition~\ref{def:FX}), and taking homotopy coinvariants of the group
action.  The key step is the construction of an equivariant Dehn invariant
(Definition~\ref{def:Di}), so that we can analyze is the total homotopy cofiber of
the original cubical diagram, prior to taking coinvariants.  We can then commute
the homotopy cofiber past the group action, and analyze them independently. We
denote this space $(\sY^X)_{hI(X)}$ (it is defined in Section~\ref{sec:large}).
Sections \ref{sec:geometry} and \ref{sec:dehn} are devoted the construction of
this cube.

In the homotopical analysis of $(\sY^X)_{hI(X)}$ a minor miracle occurs: the
space is weakly equivalent to the homotopy coinvariants of a sphere with $I(X)$
acting on it (almost) trivially. We offer two proofs of this fact in
Section~\ref{sec:twoproofs}. This allows for significant simplification of the
spectral sequences that compute its homology groups. The spectral sequences that
we use are known as the homotopy orbit spectral sequence
(Proposition~\ref{prop:hoss}) and the spectral sequence for the total homotopy
cofiber of a cube (Section~\ref{subsec:thocofib}).  $H_\ast ((\sY^X)_{hI(X)})$.
We can also use the homotopy orbit spectral sequence to compute
$H_\ast ((\sY^X)_{hI(X)})$ to obtain a shift of $H_\ast (I(X); \Q^\tw)$; this
implies Theorem~\ref{mainthm:inj}.  An analysis of the edge homomorphism of the
spectral sequence constructs the desired homomorphism in
Theorem~\ref{thm:projbase}; an explicit description of the homomorphism in the
case $\ast = n-1$ in Lemma~\ref{lem:edge-desc} implies Theorem~\ref{mainthm:a}.

\begin{rmknonum}
  In this paper we mostly focus on spherical and hyperbolic geometries, as well
  as work over $\R$ and $\CC$, as these were our main examples of interest.
  However, most of our techniques do not rely on either these choices of
  geometry or the choice of field.  In future work we hope to work out further
  implications of these approaches in other fields, geometries, and isometry
  groups. (The Euclidean case is an obvious candidate.)
\end{rmknonum}

\begin{rmknonum}
  What is especially striking about our approach is that most of the topological
  spaces we work with turn out to be homotopy-equivalent to bouquets of spheres.
  This means, in essence, that they are \emph{combinatorial} objects, rather
  than topological.  Despite this, the topological approach appears to produce
  significantly simpler proofs, and stronger results, than a purely algebraic
  one.
\end{rmknonum}

\subsection*{Organization}

In Section~\ref{sec:geometry} we introduce the basic objects of interest.
Although many of the objects and definitions are standard, several key
definitions (esp. RT-buildings) differ subtly from standard.  We have attempted
to highlight these differences in the exposition.  Section~\ref{sec:dehn}
introduces derived Dehn invariants and states that they agree with the classical
definitions; although the comparison between our objects and the classical
objects is interesting (and we believe a good introduction to simplicial
techniques) we postpone the direct comparison to Appendix~\ref{app:classical},
as it is technical and completely disjoint from the main thrust of the paper.
Section~\ref{sec:goncharov} recalls Goncharov's definition of the Dehn complex
and shows how to construct a ``geometrized'' model.  Section~\ref{sec:large} is
the main meat of the topological story: it introduces the key theorem
(Theorem~\ref{thm:reallycool}) which allows us to directly compare the homology
of the Dehn complex to the group homology of orthogonal groups.
Section~\ref{sec:conj} proves Theorems~\ref{mainthm:a} and \ref{mainthm:inj} and
explains the connection between Goncharov's original conjectures and the form in
which they arise in this paper.  Section~\ref{sec:twoproofs} proves
Theorem~\ref{thm:reallycool} and uses the proof to provide the computation for
the claim about volume in Theorem~\ref{mainthm:a}.  This section is largely
independent of much of the rest of the paper, dealing mostly with the structure
of RT-buildings.

\subsection*{Notation and conventions} \label{notation}

We work in the category of \textbf{pointed} topological spaces and simplicial
sets.  Thus homology is reduced, and all constructions on spaces are pointed.
In particular, homotopy $G$-coinvariants---denoted $\bullet_{hG}$---are taken in
a correctly-pointed manner, so that $*_{hG} \simeq *$ and
$(S^0)_{hG} \simeq (BG)_+$.  Here, $\bullet_+$ denotes adding a disjoint
basepoint, so that we can think of $S^0$ as $*_+$.  Note the difference with the
unpointed constructions: in the unpointed case $*_{hG} \simeq BG$. In general,
in order to translate from the unpointed case to the pointed case one adds a
disjoint basepoint and then works relative to that point.   All groups
in this paper are considered \textbf{discrete} unless explicitly stated
otherwise, so that $BG$ is always the Eilenberg--Mac Lane space $K(G,1)$.

Two spaces (resp. simplicial sets) are weakly equivalent if there exists a map
$f:X \to Y$ which is a bijection on connected components and such that the
induced maps $f_*: \pi_n(X,x) \to \pi_n(Y,f(x))$ are isomorphisms for all
choices of basepoint $x$ and all $n \geq 1$.  When two spaces (resp. simplicial
sets) $X$ and $Y$ are weakly equivalent, we denote this by $X \simeq Y$.  Two
simply-connected spaces are ``weakly equivalent after inverting $2$'' if there
exists a map $f: X \to Y$ such that the induced maps
$f_*: \pi_n(X,x)\otimes \Zinv \to \pi_n(Y,f(x))\otimes \Zinv$ are isomorphisms
for all $n \geq 2$.  We denote this by $X \simeq_{[2]} Y$.

The notation $X / Y$ will refer to the quotient of spaces (resp. simplicial
sets): the topological space (resp. simplicial set) given by collapsing all
points to $Y$ to a point.  The only exception to this notation will be the group
quotient $\R/\Z$ (and scalings thereof) and the group $\Z/2$.

We denote by $X^n$ either $n$-dimensional hyperbolic ($\mathcal{H}^n$) or
spherical ($S^n$) space.  In each case, we think of $X^n$ as sitting inside
$\R^{n+1}$, with subspaces being cut out by subspaces of $\R^n$ through the origin
(which intersect, respectively, the plane where $x_{n+1} = 1$, the hyperboloid
$-x_0^1+x_1^2+\cdots+x_n^2$, and the sphere in a nonempty set).  When the
dimensions is clear from context we write $X$ instead of $X^n$.

For any abelian group $A$, we write $A_\Q \defeq A\otimes \Q$.

The field $k$ is always assumed to have characteristic $0$.

\subsection*{Acknowledgements}

Both authors were supported in part by the IAS Summer Collaborators program.
The second author was supported in part by NSF CAREER-1846767.

The work in this paper is deeply indebted to Jean-Louis Cathelineau's work on
homological stability of orthogonal groups and scissors congruence
\cite{cathelineau_homology_stability,cathelineau_proj,cathelineau_bar}, which
inspired many of the approaches in this paper.

The authors would like to thank Agnes Beaudry, Sherry Gong, Richard Hain, Max
Karoubi, Cary Malkiewich, Daniil Rudenko, Jesse Silliman, Craig Westerland, and
Charles Weibel for their helpfulness and patience with our questions, especially
about regulators.  They would also like to thank Ilya Zakharevich for his
patience and willingness to read at least 4 drafts of this paper and provide
feedback which greatly improved its clarity and general exposition.

\section{RT-buildings and scissors congruence groups} \label{sec:geometry}

Our main goal in this section is to establish the basic definitions of the
objects we will be using, as many of these definitions are not (quite) standard.
Many small variations on these definitions exist in the literature (see, for
example, \cite[Chapter 2]{dupont}, \cite{cathelineau_proj}), leading to a
combinatorial explosion of choices.  In our experience only the current choices
lead to a consistent rigid derived theory.  We work over any infinite base field
$k$ of characteristic $0$.

\subsection{RT-buildings} 

\begin{definition}[Based on {\cite[Definition 1.0.3]{cathelineau_proj}}]
  A \emph{geometry over $k$}, $X$, is a vector space equipped with quadratic
  form $(E,q)$ over $k$, where $q$ is totally nondegenerate, together with its
  isometry group $I(X)$.  The \emph{dimension} of $X$ is $\dim E - 1$.  By
  definition, $I(X) = I(E)$: the subgroup of $GL(n+1;k)$ which preserves the
  quadratic form.

  When we wish to emphasize that a geometry $X$ has dimension $n$, we write it
  as $X^n$.
\end{definition}

\begin{definition} \label{def:neat}
  The \emph{neat geometries} are the \emph{spherical} geometry $S^n$, given by
  the quadratic form $x_0^2 + \cdots + x_n^2$, and the \emph{hyperbolic}
  geometry $\cH^n$, given by the quadratic form $-x_0^2 + x_1^2 + \cdots +
  x_n^2$.
  
  When it is not clear from context, we write $S^n_k$ or $\cH^n_k$ to emphasize
  that the geometries are over $k$.

  In later sections, we will often be considering maps of the form $X^n \rto X^a
  \redjoin S^b$.  A map of this sort states that we fix a type of geometry
  (spherical or hyperbolic), and both $X$'s are of this same type, of dimensions
  $n$ and $a$, respectively.
\end{definition}

\begin{definition}
  For a geometry $X = (E,q)$, where $q$ has signature $(n_-,n_+)$, a \emph{subspace}
  $U$ of $X$ is a subset of $\mathbb{P}(E)$ corresponding to a linear subspace
  $V$ of $E$ such that the restriction of $q$ to $V$ is totally nondegenerate
  and such that the signature $(m_-,m_+)$ of $q|_V$ has $m_- = n_-$.

  An \emph{angular-subspace} $U$ of $X$ is a linear subspace $V$ of $E$ such
  that the restriction of $q$ to $V$ is totally nondegenerate and such that the
  signature $(m_-,m_+)$ of $q|_V$ has $m_- = 0$.

  For any subspace or angular-subspace $U$ of $X$ we say that $U$ is
  \emph{represented} by $V$.  The \emph{dimension} of $U$ is $\dim V - 1$; if
  $\dim U = 0$ we refer to $E$ as a \emph{point} of $X$.

  When $k$ contains $\sqrt{-1}$, the condition on the signature is vacuous and
  subspaces and angular-subspaces are equivalent.
\end{definition}

\begin{remark}
  The condition on the signature may appear artificial, but it is necessary in
  order to model the types of subspaces in question.  A geometry $X$ of
  dimension $n$ can be considered to be sitting inside $k^{n+1}$ as a
  submanifold.  In the case when $k$ is not algebraically closed, a plane of
  dimension $m$ may not intersect this submanifold in a subspace of dimension
  $m-1$, as desired.  The condition on the signature ensures that this will
  happen in the cases of interest in this paper.
\end{remark}

\begin{remark}
  Many of the definitions and results in this paper will also work for the
  Euclidean geometry, as well as for geometries with signatures other than
  $(0,n+1)$ and $(1,n)$.  However, there are enough subtleties and differences
  between these cases that in this paper we focus exclusively on the spherical
  and hyperbolic cases.
\end{remark}

The key structure necessary for the program is the presence of an orthogonal
complement for any subspace and the notion of a projection onto the orthogonal
complement.

\begin{definition}
  Let $U$ be an $i$-dimensional subspace of $X$, represented by a linear
  subspace $V$ of $E$.  We define the orthogonal complement $U^\perp$ of $U$ to
  be the angular-subspace represented by $V^\perp$. 

  If $U$ is a subspace and $U'$ is an angular-subspace of $X$, represented by
  $V$ and $V'$, then we write $U \perp U'$ if $V \perp V'$.  We write
  $U \oplus U'$ for the subspace represented by $V \oplus V'$.  If $V \perp V'$
  we write $U \perp U'$ instead of $U\oplus U'$ to emphasize this fact.

  For subspaces $U \subseteq U'$ of $X$, we write
  \[\pr_{U^\perp} U' \defeq U' \cap U^\perp.\] The isometry group of
  $\pr_{U^\perp}U'$ is taken to be the subgroup of the isometry group of $U'$
  that fixes $U$.
\end{definition}

We will be using the following three properties of subspaces:
\begin{lemma}
  Let $X^n$ be a neat geometry and let $U^i$ be a subspace of $X$.  Then $\dim
  U^\perp = n-i-1$.  For a subspace $V$ containing $U$, $V$ is uniquely
  determined by $U$ and $U^\perp \cap V$.  In addition, the induced quadratic
  form on $\pr_{U^\perp} V$ has positive signature.
\end{lemma}

The key object of study in this paper is the \emph{RT-building} associated to a
geometry $X$.\footnote{The term ``RT-building'' is named after Rognes and Tits,
  as our objects are ``halfway'' between Tits' original objects---which must
  start at a nonempty subspace and end at a proper subspace---and Rognes' spaces
  $D^1(V)$, which have simplices which start at the trivial subspace and end at
  the full space.}

\begin{definition} \label{def:FX}
  Let $X$ be a geometry of dimension $n$ over $k$.
  
  Let $T_\dotp^m(X)$ be the simplicial set whose $i$-simplices are sequences
  $U_0 \subsets U_i$ of nonempty subspaces of $X$ of dimension at most $m$.  The
  $j$-th face map deletes $U_j$; the $j$-th degeneracy repeats $U_j$.  The
  isometry group $I(X)$ acts on $T_\dotp^m(X)$.

  We define the \emph{RT-building} of $X$ to be the pointed simplicial set
  given by
  \[F_\dotp^X \defeq T_\dotp^n(X)/ T_\dotp^{n-1}(X),\] with the inherited
  $I(X)$-action.  More explicitly, the non-basepoint $i$-simplices of $F_\dotp^X$ are
  sequences $U_0 \subsets U_i$, where each $U_j$ is a nonempty subspace of $X$
  and $U_i = X$.  The face maps and degeneracies work as before, with the caveat
  that if $U_{i-1} \neq X$ then $d_i$ sends the simplex $U_0 \subsets U_i$ to
  the basepoint.
\end{definition}

It turns out that the group $\tilde H_n(F_\dotp^{S^n})$ contains vital information
about scissors congruence.  In fact, this is the only nonzero homology group of
this space:
\begin{proposition} \label{prop:singleH}
  For $i \neq \dim X$, $\tilde H_n(F_\dotp^X) \cong 0$.
\end{proposition}

The fact that all homology groups above $\dim X$ are $0$ is evident from the
fact that all nondegenerate simplices have length at most $n+1$.  The fact that
all (reduced) homology groups below degree $n$ are also $0$ is more complicated;
one can refer to the Solomon--Tits Theorem \cite[Section 2]{quillen_rank}, or
use the theory developed in Appendix~\ref{app:classical}.  As the proof is
technical and not illuminating, we defer it to the appendix.

\subsection{Classical scissors congruence} We turn our attention to defining the
scissors congruence groups.  For scissors congruence to be defined we need a
notion of a geometry to work within, as well as a notion of ``inside'' and
``outside'' for polytopes; thus we will need to be working inside an ordered
field.  For now we fix $k = \R$, although most of the machinery developed should
work equally well over other ordered fields.

The basic building block of a polytope (and thus of a scissors congruence group)
is a simplex, which can be defined as a convex hull.

\begin{definition} \label{def:convhull}
  Suppose $X$ is a neat geometry of dimension $n$.
  
  A \emph{convex hull} of a tuple $(a_0,\ldots,a_m)$ of points in $X$ is any
  subset of $X$ represented by a cone over $b_i$
  \[\bigg\{\sum_{i=0}^m c_ib_i \in \R^{n+1} \,\bigg|\, c_i \geq 0 \ \forall
    i\bigg\}\]
  for any choice of $0\neq b_i\in a_i$ for all $i$.

  An \emph{$m$-simplex} in $X$ is the convex hull of a tuple $(a_0,\ldots,a_m)$
  which is not contained in an $m-1$-dimensional subspace of $X$. An
  \emph{$m$-polytope} in $X$ is a finite union of $m$-simplices; we make no
  assumptions of convexity or connectedness.  When $m=n$ we omit it from the
  terminology and refer simply to ``simplices in $X$'' or ``polytopes in $X$.''
\end{definition}

\begin{remark}
  When $X$ is hyperbolic, a simplex is uniquely determined by its vertices in
  the following sense.  The
  hyperboloid $-x_0^2+x_1^2+\cdots+x_n^2$ has two connected components, and we
  think of $X$ as one of these components and a point of $X$ as the intersection
  of the representing line with this component.  A tuple of points
  $(a_0,\ldots,a_n)$ in $X$ thus defines a tuple of vectors in $\R^{n+1}$, and
  thus the positive cone above is well-defined.

  When $X$ is spherical, there are $2^{n+1}$ possible choices of ``sign'' of the
  representatives $b_i$.  Thus a simplex is no longer uniquely defined by its
  vertices.  
\end{remark}

We can now define the scissors congruence group of $X$:
\begin{definition} \label{def:scissors} Let $X$ be a neat geometry over $\R$,
  and let $G$ be a subgroup of $I(X)$.  Then the \emph{scissors congruence group
    of $X$ relative to $G$}, denoted $\hat\P(X,G)$, is the free abelian group
  generated by polytopes in $X$ modulo the relations
  \begin{itemize}
  \item $[P\cup Q] = [P] + [Q]$ if $P \cap Q$ is contained in a finite union of
    $m-1$-dimensional subspaces.
  \item $[P] = [g \cdot P]$ for any $g\in G$.  Here $g$ acts on $P$ pointwise;
    as it is in $I(X)$ it takes convex hulls to convex hulls.
  \end{itemize}
  When $G = I(X)$ we omit it from the notation.
\end{definition}

The case $X = S^n$ is more complicated, as convex hulls are now only
well-defined up to a certain equivalence relation.  The following definition
will make all of the choices in Definition~\ref{def:convhull} equivalent.

\begin{definition} \label{def:spherangle} For any $G$-module $M$, the
  \emph{coinvariants} of $G$ acting on $M$ are defined to be the group
  \[M/(m - g\cdot m\,|\, g\in G,\ m\in M).\]
  This is isomorphic to the zeroth group homology $H_0(G,M)$.
  
  Let
  \begin{equation} 
    \Sigma: \bigoplus_{\substack{V \subseteq \R^{n+1} \\ \dim V = n}}\hat\P(V\cap
    S^n, 1) \rto \hat\P(S^n, 1)
  \end{equation}
  be the ``suspension'' map taking a simplex in $V\cap S^n$ to the union of the
  two simplices defined by the choice of representatives in $V^\perp$.  Denote
  by $\P(S^n, G)$ the cokernel of the induced map
  \[\Sigma: H_0 \bigg(G , \bigoplus_{\substack{V \subseteq \R^{n+1} \\ \dim V = n}}
    \hat\P(V\cap S^n, 1)\bigg) \rto H_0(G, \hat\P(S^n,1)).\]
  When $X \neq S^n$ we define $\P(X,G) \defeq \hat \P(X,G)$.
\end{definition}
The cokernel of $\Sigma$ turns out to be the more ``correct'' notion of scissors
congruence of the sphere, as it is most often used to measure angles.  A
subdivision of a polytope adds many angle measures that add up to the entire
sphere; thus, in order to make our definitions treat subdivisions correctly, the
entire sphere should be considered to be zero.  When we discuss the Dehn
invariant in Section~\ref{sec:dehn} this will become clearer, as Dehn invariants
are only well-defined inside these reduced scissors congruence groups.  

\begin{remark}
  The notation we are using is somewhat nonstandard.  The group $\hat\P(X,G)$ is
  usually denoted $\P(X,G)$, and the group $\P(S^n,G)$ is generally denoted
  $\tilde\P(S^n,G)$.  In this paper, however, the group of interest is
  $\P(S^n,G)$ for $X = S^n$, and we would like to unify the notation so that
  this group is the default one.  
\end{remark}

As motivation for considering scissors congruence as homotopy coinvariants we
observe that
\begin{equation} \label{eq:SCtoH0}
  \P(X,G) \cong H_0(G, \P(X,1))
\end{equation}
and that this works for the groups $\hat \P(S^n,G)$ as well.

\subsection{Geometrizing a twist}

In the classical literature on scissors congruence, flags often take the place of
polytopes (thus motivating our study of RT-buildings).  However, a complication
arises: in the algebraic story, the action of the isometry on flags is twisted
by the determinant map, so that $g \cdot [x] = (\det g)[g\cdot x]$.  In order to
construct a topological model of such a twist, we need a topological model for
``tensoring with a copy of $\Z$ with the sign action.''

\begin{definition} \label{def:Ssigma}
  Let $S^1$ be the pointed simplicial set $\Delta^1/\partial \Delta^1$.
  
  Let $S^\tw$ be the pointed simplicial set
  $\Delta^1 \cup_{\partial \Delta^1} \Delta^1$, with one of the vertices in
  $\partial \Delta^1$ taken to be the basepoint.  This is a model of a circle
  with two $0$-simplices and two $1$-simplices.  There is an action of $\Z/2$ on
  $S^\tw$ given by swapping the two $1$-simplices.
\end{definition}

\begin{remark}
  The notation $S^\tw$ is chosen to be compatible with the standard notation $M^\tw$
  for a $G$-module which is twisted by the action of a ``sign'' map $G \rto \Z/2$.
\end{remark}

The group $I(X)$ acts on $S^\tw$ via the map $\det: I(X) \rto \Z/2$.  Note
that $H_n(F_\dotp^X) \cong H_{n+1}(S^\tw\smash F_\dotp^X)$ as groups.  As $I(X)$-modules,
these differ only by the action on $S^\tw$, which adds a twist $\dotp^\tw$ by the
determinant.  In particular, this means that
\[H_0(G, H_n(F_\dotp^X)^\tw) \cong H_0(G, H_{n+1}(S^\tw\smash F_\dotp^X)).\] In
other words, the $I(X)$-coinvariants of $H_{n+1}(S^\tw \smash F_\dotp^X)$ are
exactly the ``$I(X)$-semi-coinvariants'' in $H_n(F_\dotp^X)$. From the homotopy
orbit spectral sequence (see Proposition~\ref{prop:hoss}), we have
\[H_0(G, H_{n+1}(S^\tw\smash F_\dotp^X)) \cong H_{n+1}((S^\tw\smash
  F_\dotp^X)_{hG}).\]

\begin{remark} \label{rem:dimshift} It may seem that the approach of
  ``geometrizing'' $\dotp^\tw$ by smashing with $S^\tw$ produces a spurious
  increase of dimension.  However, this increase is present the algebraic story
  (discussed in \cite{sah79}, \cite{goncharov}, and others) as well: the
  scissors congruence groups for $S^{n-1} \subseteq \R^n$ must be graded by the
  ambient dimension $n$, rather than $n-1$, in order to make the Dehn invariant
  a graded homomorphism.  In addition, Section~\ref{sec:conj} shows that these
  dimensions allow the maps from $K$-theory to have the correct grading.  It is
  thus unsurprising that something of this sort should appear in the topological
  viewpoint.
\end{remark}

\subsection{The geometrization of scissors congruence groups}

We now state the connection between scissors congruence and RT-buildings:
\begin{theorem} \label{thm:f<>p} Suppose $k = \R$.  Let $X$ have dimension $n$
  and let $G$ be a subgroup of the isometry group of $X$.  For a neat geometry
  $X$, 
  \[\P(X,G) \cong H_{n+1}((S^\tw \smash F_\dotp^{X})_{hG}) \cong
    H_0(G, H_{n+1}(S^\tw\smash F_\dotp^{X})).\]
  This map is induced by the $G$-equivariant map
  $\P(X,1) \to H_{n+1}(S^\sigma \smash F_\dotp^X)$ which takes a simplex
  with vertices $\{x_0,\ldots,x_n\}$ to the sum
  \[ \sum_{\sigma \in \Sigma_{n+1}} \sgn(\sigma)
    [\rspan(x_{\sigma(0)}) \subseteq \rspan(x_{\sigma(0)},x_{\sigma(1)})
    \subsets \rspan(x_{\sigma(0)},\ldots,x_{\sigma(n)})].\]
\end{theorem}
As the proof of this theorem is technical and not illuminating, we postpone it
until Appendix~\ref{app:classical}.  A similarly-simple model for $\hat\P(S^n,G)$
is not known.

The theorem above implies that scissors congruence information is contained
inside $H_{n+1}(S^\tw \smash F^X_\dotp) \cong H_n(F^X_\dotp)^\tw$.  The value
added by the topology is that when $G \neq 1$ the space $(S^\tw\smash
F_\dotp^X)_{hG}$ contains nontrivial higher homological information, and thus
remembers more about the algebra of $G$ than the left-hand side.

Inspired by Theorem~\ref{thm:f<>p} we can now define generalized scissors
congruence groups:
\begin{definition} \label{def:gen-scissors} Let $X$ be a geometry over $k$, and
  let $G \leq I(X)$.  The \emph{scissors congruence group of $X$}, written
  $\P(X,G)$, is
  \[\P(X,G) \defeq H_{n+1}((S^\tw \smash F_\dotp^X)_{hG}).\]
  When $G = I(X)$ we omit it from the notation.
\end{definition}

\begin{remark} \label{rmk:againJustify} At this point it may be tempting to think
  that since all spaces under consideration are simply-connected, the current
  topological model can contain no information that is not contained in
  $\P(X,G)$.  However, this misses the important point that we are keeping track
  not only of the group, but also of the $G$-action.  Taking orbits on the level
  of homology is the ``underived'' model, which cannot keep track of higher
  homotopical information.  Taking homotopy coinvariants will remember this
  higher information, analogously to the way that taking homotopy coinvariants
  of $G$ acting on a point produces $BG$, which has many interesting homology
  groups (even though a point does not).

  A large part of the value in these models is that we can define a \emph{rigid}
  topological model of the Dehn invariant. This allows us to use simple
  topological techniques to prove Theorem~\ref{thm:reallycool}, none of which
  work in the case where we are working up to homotopy.
\end{remark}

As a first application of the topological viewpoint, we show that when $X =
S^{2n}$ there is no interesting scissors congruence information:
\begin{proposition} \label{prop:evencontr}
  When $n \geq 0$, after inverting $2$, 
  \[(S^\tw \smash F_\dotp^{S^{2n}})_{hO(2n+1)} \simeq_{[2]} *.\]
  In other words, the left-hand side is connected and for $i \geq 1$,
  \[\pi_i((S^\tw \smash F_\dotp^{S^{2n}})_{hO(2n+1)}) \otimes \Zinv \cong H_i\left((S^\tw \smash
    F_\dotp^{S^{2n}})_{hO(2n+1)};\Zinv\right) \cong 0.\]
  In particular, for $n > 0$, $\P(S^{2n}) = 0$.
\end{proposition}

\begin{proof}
  First consider the case when $n > 0$.  The matrix $-I \in O(2n+1)$ acts on all
  homology groups of $S^\tw \smash F_\dotp^{S^{2n}}$ by $-1$.  Thus by the
  ``Center Kills Lemma'' \cite[Lemma 5.4]{dupont}, $2$ annihilates
  $H_i(O(2n+1), \tilde H_{2n}(S^\tw \smash F_\dotp^{S^{2n}}; \Zinv ))$ for all
  $i$; since we have inverted $2$, these must all be $0$.  By the homotopy orbit
  spectral sequence (Proposition~\ref{prop:hoss}),
  \[\tilde H_*\left((S^\tw \smash F_\dotp^{S^{2n}})_{hO(2n+1)}; \Zinv\right) = 0\] for all $i$.
  Since the simplicial set $(S^\tw \smash F_\dotp^{S^{2n}})_{hO(2n+1)}$ is
  simply-connected (as suspensions and homotopy coinvariants commute), it must
  be contractible after inverting $2$.

  The last part of the proposition follows because by \cite[Corollary
  2.5]{dupont}, $\P(S^{2n}_\R)$ is $2$-divisible, so inverting $2$ does not
  affect the lowest homology group.

  When $n =0$ the situation is somewhat more complicated.  In this case,
  $F_\dotp^{S^0} = S^0$, with one non-basepoint $0$-simplex represented by the
  subspace $S^0$ and no other nondegenerate simplices.  Then $S^\sigma \smash
  S^0 = S^\sigma$, with $O(1) = \Z/2$ acting on it via the sign representation.
  By definition, $(S^\sigma)_{h\Z/2} \cong B\Z/2$, which is a nilpotent space.
  Thus after inverting $2$ its homotopy groups are trivial, as desired.

  The final statement in the proposition is a direct consequence of
  \cite[Proposition 6.2.2]{sah79}: since $\P(S^{2n})$ is $2$-divisible inverting
  $2$ does not affect the scissors congruence group and the original group
  itself must be $0$.
\end{proof}

In general, the homology of $(S^\sigma \smash F_\dotp^X)_{hG}$ can be described
in terms of the homology of $G$ and the homology of $F_\dotp^X$:
\begin{lemma} \label{lem:Hconinv}
  Let $X$ be a neat geometry of dimension $n$.  For all $i$,
  \[\tilde H_i\big((S^\sigma\smash F_\dotp^X)_{hG}) \cong H_{i-(n+1)}\big(G, \tilde
    H_{n}(F_\dotp^X)^\tw\big),\] where $\cdot^\tw$ denotes that the action
  of $G$ on the homology is twisted by multiplication by the determinant.  In
  particular, if $i-(n+1)$ is negative the left-hand side is $0$.
\end{lemma}

\begin{proof}
  This follows directly from the homotopy orbit spectral sequence (see
  Proposition~\ref{prop:hoss}).  Since the reduced homology of
  $S^\sigma \smash F_\dotp^X$ is concentrated in degree $n+1$, the homotopy
  orbit spectral sequence is contained in the $n$-th column, and thus collapses.
  This implies that
  \[\tilde H_i((S^\sigma \smash F_\dotp^X)_{hG}) \cong H_{i-(n+1)}(G, \tilde
    H_{n+1}(S^\sigma\smash F^X_\dotp)).\] Using that
  $\tilde H_{n+1}(S^\sigma\smash F^X_\dotp) \cong \tilde
  H_n(F^X_\dotp)^\tw$ as a $G$-module completes the proof.
\end{proof}

\section{Rigid derived Dehn invariants} \label{sec:dehn}

The statement (rephrased in modern terminology) of Hilbert's third problem is
extremely simple:
\begin{quotation}
  Do there exist two polyhedra with the same volume which are not scissors
  congruent? 
\end{quotation}
The answer, given in 1901 by Dehn is ``yes'': the cube and regular tetrahedron
are not scissors congruent, even if they have the same volume.  Dehn proved this
statement by constructing a second invariant of polyhedra (these days called the
``Dehn invariant'') which is zero on a cube and nonzero on any regular
tetrahedron.  This invariant takes values in $\R\otimes_\Z \R/\Z$---a difficult
group to work in, but even more startling given that tensor products were only
originally defined in 1938.  In 1965, Sydler proved that the volume and the Dehn
invariant uniquely determine scissors congruence classes; phrased in a more
modern fashion (after \cite{jessen68}), this is equivalent to stating that the
volume map is injective when restricted to the kernel of the Dehn invariant.

\subsection{The classical story: constructing an equivariant Dehn invariant}
In this section we give a definition of the classical Dehn invariant (extended
to arbitrary dimensions in the form proposed by Sah in \cite{sah79}) and
construct a derived model (a ``geometrization'').  The homological inspiration
for our construction is Cathelineau's approach in \cite{cathelineau_bar,
  cathelineau_proj, cathelineau_homology_stability}.

\begin{definition}
  Let $X^n$ be a neat geometry over $\R$, and consider $\P(X)$.  For any integer
  $0 < i < n$, we define the \emph{$i$-th classical Dehn invariant} in the
  following manner.  Since $\P(X)$ is generated by simplices, it suffices to
  define it on simplices.  For a simplex $\sigma$ in $X$ with vertices
  $\{x_0,\ldots,x_n\}$, we define
  \[\hat D_i(\sigma) = \sum_{\substack{J \sqcup J' = \{0,\ldots,n\} \\ |J| = i+1 \\ U
        = \mathrm{span}\ x_J}} [x_J] \otimes [\pr_{U^\perp} (x_{J'})] \in
    \P(X^i, I(X^i)) \otimes \P(S^{n-i-1}, I(S^{n-i-1})).\] Here, $x_J$ is
  the set $\{x_j \,|\, j\in J\}$, $[x_J]$ is the class of the simplex with
  vertices $x_J$ in an isometric copy of $X^i$ sitting inside
  $X^n$, and $[\pr_{U^\perp}x_{J'}]$ is the class in $\P(S^{n-i-1},
  I(S^{n-i-1}))$ of the simplex spanned by the projections of the
  $x_{J'}$.  For a more detailed discussion of this, see \cite[Section
  6.3]{sah79}.
\end{definition}

This generalization of the Dehn invariant allows for the following
question, the ``generalized Hilbert's third problem'':
\begin{question}[Generalized Hilbert's third problem]
  In a neat geometry $X^n$, is volume injective when restricted to the kernel of
  $\bigoplus_{i=1}^n \hat D_i$?
\end{question}

\begin{remark} \label{rem:spherVol}
  There is an important subtlety which is often overlooked in the definition of
  scissors congruence groups.  Although volume is well-defined on $\P(\cH^n)$
  and $\hat \P(S^n)$,\footnote{and also $\P(E^n)$, although $E^n$ is not a neat
    geometry} it is \emph{not} well-defined on $\P(S^n)$: inside $\P(S^n)$ we
  quotient out by \emph{lunes}, which are polytopes which are ``suspensions'' of
  lower-dimensional polytopes.  Since lunes of any volume can be constructed,
  the volume map on $\hat \P(S^n)$ does not descend to a well-defined map out of
  $\P(S^n)$ for $n > 1$.  (When $n=1$ all lunes are semicircles, and thus length
  is well-defined mod $\pi$.)
\end{remark}

The Dehn invariant is not well-defined if $\P(X)$ is replaced by $\P(X,1)$, as
there is no natural way to consider $[x_I]$ as sitting inside $\P(X^i, 1)$.
Since the goal is to postpone taking the coinvariants of $I(X)$ for as long as
possible in order to model $D_i$ on an RT-building, it is necessary to construct
the Dehn invariant as the functor of $I(X)$-coinvariants applied to an
equivariant homomorphism of $I(X)$-modules; this motivates the later geometric
construction.  To define a map
\[\hat D_U: \P(X, 1) \rto \P(U, 1) \otimes \P(U^\perp,1)\] consider  a
simplex $\sigma$ with vertices $\{x_0,\ldots,x_n\}$ in $X$, and suppose that
$U = \mathrm{span}(x_0,\ldots,x_i)$.  Let $\tau$ be the projection of
$\{x_{i+1},\ldots,x_n\}$ to $U^\perp$.  Define
\[\hat D_U([x_0,\ldots,x_n]) \defeq [x_0,\ldots,x_i] \otimes [\tau].\]
For any simplex $\{x_0,\ldots,x_n\}$ such that there do not exist $0 \leq j_0 <
\cdots < j_i \leq n$ with $U = \mathrm{span}(x_{j_0},\ldots,x_{j_i})$, define
\[\hat D_U([x_0,\ldots,x_n]) \defeq 0.\]

\begin{lemma} \label{lem:componentDehn}
  With this definition,
  \[\setlen{4em}{\P(X,1) \rto^{\bigoplus \hat D_U} \bigoplus_{\substack{U \subseteq X \\ \dim U = i}} \P(U,1)
    \otimes \P(U^\perp,1)}\]
is well-defined and $I(X)$-equivariant.  After taking $I(X)$-coinvariants this
map becomes $\hat D_i$.
\end{lemma}

\begin{proof}
  To prove well-definedness it suffices to show that for any simplex $\sigma$
  all but finitely many of the $\hat D_U$ are $0$.  This is true because there
  are only finitely many subspaces $U$ which are the span of a subset of
  $\{x_0,\ldots,x_n\}$.
  
  The action of $I(X)$ on the left is simply an action on tuples.  The action on
  the right is a bit more complicated: it acts on the indices of the sum, and
  acts within each group, as well.  However, simplices in $U$ can be thought of
  as $i$-simplices in $X$ that happen to be contained in $U$, on each individual
  simplex the action is via acting on each vertex of the simplex; in this way
  the right-hand side is considered to be sitting inside
  $\bigoplus_{\substack{U \subseteq X\\ \dim U = i}} \P(X,1) \otimes \P(X,1)$.

  It remains to check that taking $I(X)$-coinvariants makes this map into
  $\hat D_i$.  The left-hand side becomes $\P(X)$.  Moreover, $I(X)$ identifies
  all of the summands on the right-hand side, and the stabilizer of any fixed
  $U$ is $I(U) \times I(U^\perp)$; thus the right-hand side is
  $H_0(I(U)\times I(U^\perp), \P(U,1)\otimes \P(U^\perp,1))$.  We have an
  isomorphism \cite[Section V.2]{brown_cohomology}
  \[H_0(I(U), \P(U,1))\otimes H_0(I(U^\perp), \P(U^\perp,1)) \rto H_0(I(U)
    \times I(U^\perp), \P(U,1) \otimes \P(U^\perp,1)),\] induced by the
  cross-product in homology, as the group $\P(U,1)$ is free (by the
  Solomon--Tits theorem \cite[Section 2]{quillen_rank}).  Thus the right hand
  side is $\P(U) \otimes \P(U^\perp)$, as desired.

  To see that the map is exactly $\hat D_i$, note that taking the $I(X)$-coinvariants
  adds up the images of all nonzero $\hat D_U$ on a given simplex $\sigma$; this is
  exactly the definition of $\hat D_i$.
\end{proof}

The classical Dehn invariant can be iterated in the following sense.  Suppose
that $i < j$; then the following square commutes: 
\begin{general-numbered-diagram}{2.5em}{4em}{diag:dehnsquare}
  { \P(X^n) & \P(X^i) \otimes \P(S^{n-i-1}) \\
    \P(X^j) \otimes \P(S^{n-j-1}) & \P(X^i) \otimes \P(S^{j-i-1})
    \otimes \P(S^{n-j-1}). \\};
  \arrowsquare{\hat D_i}{\hat D_j}{\id\otimes \hat D_{j-i}}{\hat D_{i}\otimes \id}
\end{general-numbered-diagram}
Goncharov uses this observation to construct a chain complex of Dehn invariants
which he conjectures is related to algebraic $K$-theory.  For a discussion of
this, see Section~\ref{sec:goncharov}.

\subsection{The derived construction}
The goal is to construct a notion of the Dehn invariant on $F_\dotp^X$ which
will produce the classical Dehn invariant when $k = \R$, but only \emph{after}
taking coinvariants and homology.  The idea that the Dehn invariant should be
constructed in this manner is key to making the analysis in
Section~\ref{sec:large} possible.  It is also the only perspective from which it
appears to be possible to construct the derived Dehn invariant; the authors
attempted many non-equivariant constructions before settling on this approach.
That this makes the Dehn invariant concise and clean and exposes its
combinatorial nature is a minor miracle.

The key idea here is to replace the tensor product of abelian groups with the
\emph{reduced join} of simplicial sets.

\begin{remark}
  Classically, the tensor product would be replaced by the smash product, and
  choosing instead the reduced join may appear to be a perverse choice: the
  reduced join of simplicial sets is not symmetric in the category of simplicial
  sets, and using it to model a symmetric structure like the tensor product
  feels unnatural.  And, as above, the authors spent considerable time on
  attempts to rework this material using a smash product.  Unfortunately (or
  perhaps incredibly interestingly), it does not seem possible to construct a
  topological model of the Dehn invariant using a smash product of
  spaces.\footnote{Indeed, it seems that this may not be possible with any
    symmetric notion of product.}  (See also Remark~\ref{rem:dimshift}.)

  An interesting corollary of this is that the constructions in this section are
  fundamentally \emph{unstable}.  Smash products of spaces lift naturally to
  smash products of spectra, and therefore give some hope that analogous
  constructions could be lifted to stable models of scissors congruence (such as
  those arising from \cite{CZ-cgw, Z-Kth-ass}).  Unfortunately, this does not
  appear to be the case, and the natural questions arise: how stable is the Dehn
  invariant?  What parts of it can be seen stably?  And which portions are
  irredeemably unstable?
\end{remark}

\begin{definition}
For pointed simplicial sets $X$ and $Y$, the \emph{reduced join} $X\redjoin Y$ is defined by
\[(X \redjoin Y)_m = \bigvee_{i+j=m-1} X_i \smash Y_j.\] For a simplex
$(x,y)\in X_i \smash Y_j$, the face maps $d_\ell$ are defined to be
$d_m \times 1: X_i \smash Y_j \rto X_{i-1} \smash Y_j$ when $\ell \leq i$, and
$1 \times d_{\ell-i-1}: X_i \smash Y_j \rto X_i\smash Y_{j-1}$ otherwise.  If
$i = \ell = 0$ or $j = m-1-\ell = 0$ then the face map takes the simplex to the
basepoint.  Degeneracies are defined analogously, with the first $i+1$ acting on
the $x$-coordinate, and the last $m-i-1$ acting on the $y$-coordinate.  Note
that this structure makes the reduced join \emph{asymmetric}.
\end{definition}

For those unfamiliar with reduced joins, an introduction and proofs of the most
relevant properties of the reduced join are in Section~\ref{sec:redjoin}.  The
most important feature of reduced joins is their relationship to smash
products; the proof is given in Section~\ref{sec:redjoin}:

\begin{lemma}[Lemma~\ref{lem:smashToJoin}] \label{lem:sma->join} Let $X$ and $Y$
  be pointed simplicial sets.  There exists a simplicial weak equivalence
  \[S^1 \smash X \smash Y \rto X \redjoin Y.\]
\end{lemma}

To define Dehn invariants on $F_\dotp^X$ (Definition~\ref{def:FX}) the first
step is, as above, to define a Dehn invariant indexed by a single subspace.

\begin{definition}
  Let $U$ be a proper nonempty subspace of $X$.  Define the \emph{rigidly
    derived Dehn invariant relative to $U$,
    $D_U: F_\dotp^X\rto F_\dotp^U\redjoin F_\dotp^{U^\perp}$} by
  \[D_U(U_0 \subsets U_n) = \begin{cases} * \caseif \nexists\ j\hbox{ s.t. }U_j=U \\
      (U_0 \subsets U_j) \smash (\pr_{U^\perp}U_{j+1}\subsets \pr_{U^\perp}U_n) \caseif
      j=\max\{i \,|\, U_i=U\}. \end{cases}\]
  We call a $j$ as above the \emph{$U$-pivot} of $U_0 \subsets U_n$.  
\end{definition}

That $D_U$ is compatible with the simplicial structure is direct from the
definitions.  
Observe that if a $U$-pivot exists then it is strictly less than $n$, since
$U \neq X$.

\begin{lemma}
  Let $U$ be a proper nonempty subspace of $X$, and write $I(X,U)$ for the
  subgroup of $I(X)$ of those elements fixing $U$.  The group $I(X,U)$ acts on
  $U^\perp$ and $D_U$ is $I(X,U)$-equivariant.
\end{lemma}

\begin{proof}

  The first claim follows from the definition of orthogonal complement.
  To check equivariants it suffices to check that for any $g\in I(X,U)$, the map
  $D_U$ commutes with the action of $g$. If a $U$-pivot exists then the action
  of $g$ passes to $F_\dotp^{U^\perp}$, and thus commutes with $D_U$.  If no
  $U$-pivot exists then this is also true after applying $g$; since $g$ fixes
  the basepoint the action of $g$ commutes with
  $D_U$. 
\end{proof}

This derived Dehn invariant can also be iterated on the nose, analogously to
\ref{diag:dehnsquare}):

\begin{lemma}
  Let $U \subsetneq V$ be proper nonempty subspaces of $X$.  Then the following
  diagram commutes:
  \begin{diagram}[6em]
    { F_\dotp^X & F_\dotp^U\redjoin F_\dotp^{U^\perp} \\
      F_\dotp^V \redjoin  F_\dotp^{V^\perp} & F_\dotp^U \redjoin  F_\dotp^{U^\perp \cap V} \redjoin  F_\dotp^{V^\perp} \\};
    \arrowsquare{D_U}{D_V}{\id\redjoin D_{U^\perp\cap V}}{D_U \redjoin  \id}
  \end{diagram}
\end{lemma}

\begin{proof}
  Fix any $m$-simplex $U_0 \subsets U_m$ with $U$-pivot $i$ and $V$-pivot $j$.
  For any subspace $W$ of $X$, if $\pr_{U^\perp}(W) \subseteq V \cap U^\perp$
  then we must have $W \subseteq V$.  In particular,
  \begin{align*}
    &((1\redjoin D_{U^\perp\cap V})\circ D_U)(U_0 \subsets U_m)\\
    &\qquad= (1\redjoin D_{U^\perp\cap V}) ((U_0\subsets U_i)\smash
      (\pr_{U^\perp}(U_{i+1})\subsets \pr_{U^\perp}(U_m))) \\
    &\qquad= (U_0 \subsets U_i) \smash (\pr_{U^\perp}(U_{i+1}) \subsets
      \pr_{U^\perp}(U_j)) \smash (\pr_{V^\perp}\pr_{U^\perp}(U_{j+1}) \subsets
      \pr_{V^\perp}\pr_{U^\perp}(U_m)) \\
    &\qquad= (U_0 \subsets U_i) \smash (\pr_{U^\perp}(U_{i+1}) \subsets
      \pr_{U^\perp}(U_j)) \smash (\pr_{V^\perp}(U_{j+1}) \subsets
      \pr_{V^\perp}(U_m)) 
  \end{align*}
  where the last step follows because $V^\perp \subseteq U^\perp$.  This is
  equal to the composition around the bottom, as desired.
\end{proof}

Up to this point, the definitions and results can be constructed for the smash
product, instead of the reduced join.  However, the authors could not find
anything analogous to the definition below when $\redjoin$ is replaced by
$\smash$:

\begin{definition} \label{def:Di}
  Let $0 < i < n$.  Define the \emph{dimension-$i$ derived Dehn
    invariant} $D_i$ to be the lift in the following diagram:
  \begin{diagram}[8em]
    { & \bigvee_{\substack{U \subseteq X\\ \dim U = i}}
      F_\dotp^U\redjoin F_\dotp^{U^\perp} \\
      F_\dotp^X &  \prod_{\substack{U \subseteq X\\ \dim U = i}}
      F_\dotp^U\redjoin F_\dotp^{U^\perp}\\};
    \to{2-1}{2-2}^{\prod D_U}
    \cofib{1-2}{2-2}
    \diagArrow{densely dotted,->}{2-1}{1-2}^{D_i}
  \end{diagram}
  This is well-defined: every simplex contains at most one space of dimension
  $i$, and thus only a single dimension-$i$ component will be nontrivial on
  it.  
\end{definition}

\begin{lemma}
  $D_i$ is well-defined and $I(X)$-equivariant.
\end{lemma}

This produces a Dehn invariant for a fixed dimension.  Moreover, this Dehn
invariant can be put into a square similar to (\ref{diag:dehnsquare}).  When
$i<j$ the diagram
\begin{diagram-numbered}[4.4em]{diag:derdehnsquare}
  { F_\dotp^X & \bigvee_{\substack{U \subseteq X\\ \dim U = i}} F_\dotp^U \redjoin
    F_\dotp^{U^\perp} \\
    \bigvee_{\substack{V \subseteq X\\ \dim V = j}} F_\dotp^V \redjoin
    F_\dotp^{V^\perp} & \bigvee_{\substack{U \subseteq V \\ \dim U = i \\ \dim V = j}}
    F_\dotp^U\redjoin F_\dotp^{U^\perp\cap V} \redjoin F_\dotp^{V^\perp}
    \\};
  \arrowsquare{D_i}{D_j}{\id\redjoin D_{j-i}}{D_{i} \redjoin \id}
\end{diagram-numbered} commutes and is rigidly $I(X)$-equivariant.

Analogously to Theorem~\ref{thm:f<>p}, this construction is compatible with the
classical story; as before the proof is postponed to
Appendix~\ref{app:classical}.
\begin{theorem} \label{thm:RDehneq} Suppose $\dim X = n$ and $k = \R$.  Then
  $H_0(I(X); H_{n+1}(S^\tw \smash D_i))$ is the classical Dehn invariant.
\end{theorem}

\section{A geometrization of the Dehn complex} \label{sec:goncharov}

Let $X$ be a neat geometry (in the sense of Definition~\ref{def:neat}).  

In \cite{goncharov} Goncharov considers a complex $\P_*(X)$ constructed out of
iterations of the Dehn invariant, and gives several conjectures relating these
to algebraic $K$-theory.  These conjectures will be discussed in
Section~\ref{sec:conj}; here we focus on the construction of this complex and
its geometrization.  

We begin with an informal outline in the case $k = \R$ and $\dim X = 2n-1$.
Using the square (\ref{diag:dehnsquare}) the classical Dehn invariant $\hat D_i$
can be iterated by varying over all possible values of $0 < i < 2n-1$; this
produces a commutative cube of dimension $2n-2$ whose vertices contain tensor
products of scissors congruence groups.  When $j$ is even, $\P(S^j) = 0$
(Proposition~\ref{prop:evencontr}); removing the coordinates where these appear
leaves an $(n-1)$-dimensional cube.  Goncharov considers the total complex of
this cube in \cite{goncharov}; we refer to this complex as the \emph{Dehn
  complex} and denote it by $\P_\ast(X)$.  One advantage is that it allows for
the following rephrasing of the generalized Hilbert's third problem for reduced
spherical and hyperbolic scissors congruence groups:
\begin{question}  Is volume injective on $H_{n-1}\P_*(X^n)$?
\end{question}

The goal of this section is to develop a tool for analyzing this complex using
total homotopy cofibers of cubical diagrams; as an additional benefit, a
definition of the Dehn complex for arbitrary fields $k$ naturally emerges.

More formally:

\begin{definition}
  Let $I$ be the category $0 \rto 1$.  An \emph{$n$-cube} in $\C$ is a functor
  $I^n \to \C$.  Suppose that $\C$ is a model category.\footnote{Model
    categories are just one of a wide variety of situations (often called
    ``homotopical categories'') in which it is possible to define homotopy
    cofibers (which is all we need for the current application). For an overview
    of this, see for example \cite{riehl_homotopical}.}  Write $\tilde I^n$ for
  the full subcategory of $I^n$ which does not contain the object
  $(1,\ldots,1)$.
  
  Let $F: I^n \rto \C$ be any functor.  The \emph{total homotopy cofiber} $\tcofib
  F$ is the homotopy cofiber of the map
  \[\colim^h F|_{\tilde I^n} \rto F(1,\ldots,1).\]
\end{definition}
For a more in-depth discussion of the total homotopy cofiber, see \cite[Section
5.9]{munsonvolic}.

The important examples are the following:
\begin{example}
  In the case $n=1$ the cube $F$ becomes a morphism $M \rto M'$ of $R$-modules,
  which can be thought of as a morphism of chain complexes concentrated in
  degree $0$.  Taking the homotopy cofiber produces
  \[\makeshort{\hocofib(M[0] \rto M'[0]) = (0 \rto M \rto M' \rto 0)},\]
  with $M'$ in degree $0$ and $M$ in degree $1$.  Tautologically, this is the
  total complex of the $1$-complex given by the original $1$-cube.
\end{example}

\begin{example}\label{example:cube}
  Now consider the general case.  Let $F: I^n \rto \Mod_R$ be a functor; this is
  an $n$-complex which has length $2$ in each direction.
  One can check that the total complex of $F$ is quasi-isomorphic to $\tcofib F[0]$.
\end{example}

To construct the Dehn complex it is more convenient to work with a slightly
different coordinatization of a cube.

\begin{definition} \label{def:I}
  Denote by $\I_d$ the category whose objects are sequences
  $\vec A = (b,a_1,\ldots,a_i)$ of nonnegative integers such that
  $b+a_1+\cdots+a_i = d$ and in which all $a_j$ are positive and even.  There
  exists a morphism $(b,a_1,\ldots,a_i) \rto (b',a_1',\ldots,a_\ell')$ if there
  exist indices $0\leq i_0 < \cdots < i_\ell = i$ such that
  $b = b' + a_1' + \cdots + a_{i_0}'$ and
  $a_j = a_{i_{j-1}+1}' + \cdots + a_{i_j}'$.

  Note that $\I_d$ is an $\lfloor\frac{d-1}{2}\rfloor$-cube via the map
  $(b,a_1,\ldots,a_i) \to
  (\delta_1,\ldots,\delta_{\lfloor\frac{d-1}{2}\rfloor})$, where $\delta_j = 1$
  if there exists an index $\ell$ with $b+a_1+\cdots+a_\ell = 2j$ and $0$
  otherwise.
\end{definition}

\begin{definition} \label{def:dehncomplex} Let $X$ be a neat geometry of
  dimension $d$.

  Define the \emph{Dehn complex} $\P_*(X)$ to be the total complex
  (equivalently, total homotopy cofiber) of the cube $\DD: \I_d \rto \AbGp$
  sending $(b,a_1,\ldots,a_i)$ to
  \[\DD(b,a_1,\ldots,a_i) = \Zinv\otimes \P(X^b) \otimes \bigotimes_{j=1}^i
    \P(S^{a_j-1})\]
  and the map $(b,a_1,\ldots, a_j+a_{j+1},\ldots,a_i)
  \rto (b,a_1,\ldots,a_i)$ to $1\otimes\cdots\otimes D_{a_j} \otimes
  \cdots\otimes 1$.  

  The \emph{equivariant Dehn cube} is the cube $\DD^\eqvt:\I_d \rto \Mod_{I(X)}$
  given by 
  \[\DD^\eqvt(b,a_0,\ldots,a_i) = \Zinv\otimes \bigoplus_{\substack{W \oplus^\perp
        V_1 \oplus^\perp\cdots \oplus^\perp V_i =X \\ \dim W = b \\ \dim V_j =
        a_j-1}} \P(W,1) \otimes \bigotimes_{j=1}^i \P(V_j,1).\] By
  the same reasoning as in the proof of Lemma~\ref{lem:componentDehn}, we see
  that $H_0(I(X), -)\circ  \DD^\eqvt = \DD$.\footnote{The only change that
    tensoring with $\Zinv$ imposes is that when $n = 0$, $\P(X^n) \cong \Zinv$,
    instead of $\Z$ (as it would usually be: it is a count of the number of
    points).  All other classical scissors congruence groups are $2$-divisible.}
\end{definition}

Thus the Dehn complex is obtained by constructing a cube of coinvariants of
homology groups and taking its total homotopy cofiber.  The goal of this section
is to construct, $I(X)$-equivariantly, a ``geometrization'': a cube of spaces
that produces $\DD$ after taking homology and then the $I(X)$-coinvariants.

\begin{remark} As mentioned in (\ref{eq:SCtoH0}) and Theorem~\ref{thm:f<>p},
  taking coinvariants in the homology can be replaced by taking the homotopy
  coinvariants of an action on a space.  Since homotopy coinvariants and the
  total homotopy cofiber commute past one another, in future sections these are
  applied in the opposite order to relate the homology of the Dehn complex to
  algebraic $K$-theory.
\end{remark}

We proceed as in previous sections: by replacing $\P(X)$ with $F_\dotp^X$.  It
cannot be done over $\Z$, but instead over $\Zinv$; the difficulties are
highlighted in the differences between $F_\dotp^{\vec A}$ and
$J_\dotp^{\vec A}$:

\begin{definition}\label{def:flagspace}
  Let $\vec A$ be any tuple of integers $\vec A = (b,a_1,\ldots,a_i)$.  
  Define
  \[F_\dotp^{\vec A} \defeq \bigvee_{\substack{W\oplus^\perp \bigoplus^\perp V_j
        = X \\ \dim V_j = a_j -1 \\ \dim W = b}} F_\dotp^W \redjoin
    \bigjoin_{j=1}^i F_\dotp^{V_j}\] (with 
  $\redjoin$-factors ordered from left to right) and
  \[J_\dotp^{\vec A} \defeq \bigvee_{\substack{W\oplus^\perp \bigoplus^\perp
        V_j = X
        \\ \dim V_j = a_j -1 \\ \dim W = b}}
    F_\dotp^W \smash  \bigwedge_{j=1}^i (S^\tw \smash F_\dotp^{V_j}).\]

\end{definition}

The construction of the Dehn complex in spaces can be duplicated in the current context:
\begin{definition} \label{def:Dehnspace} Define the functor
  $\YY: \I_d \rto \Top$ by
  \[\vec A \rgoesto S^\tw \smash F_\dotp^{\vec A},\]
  with morphisms given by the appropriate $D_i$.  Define the \emph{Dehn
    space} $\sY^{X}$ by
  \[\sY^{X} =  \tcofib \YY.\]
\end{definition}

\begin{theorem} \label{thm:gonchcomp} Let $X$ be a neat geometry of dimension
  $d$. The Dehn complex is quasi-isomorphic to the total complex of the
  $\lfloor\frac{d-1}{2}\rfloor$-cube given by
  \[H_{d+1}\left(\YY(-)_{hI(X)}; \Zinv\right): \vec A \rgoesto H_{d+1}\left((S^\tw
    \smash F_\dotp^{\vec A})_{hI(X)};\Zinv\right).\]
\end{theorem}

The rest of this section is dedicated to the proof of this theorem; as
everything from this point on will be done with $\Zinv$ coefficients, the
coefficients are omitted from the notation.  From the homotopy orbit spectral
sequence (Proposition~\ref{prop:hoss}) applied to the right-hand side of the
given formula, it suffices to construct a natural isomorphism
$\DD \Rto H_0(I(X),H_{d+1}(\YY))$.  Because $\DD \cong H_{0}(I(X), \DD^\eqvt)$ it
suffices to produce an $I(X)$-equivariant natural isomorphism
$\alpha: \DD^\eqvt \Rto H_{d+1} \circ \YY$.

To produce the value $\alpha_{\vec A}$ of $\alpha$ on $\vec A$, first observe that 
\begin{align*}
  \DD^\eqvt(\vec A) &=  \Zinv\otimes \bigoplus _{\substack{W\oplus^\perp \bigoplus^\perp
        V_j = X
        \\ \dim V_j = a_j -1 \\ \dim W = b}} \P(W,1)\otimes\bigotimes_{j=1}^i
  \P(V_j,1)  \\ & =
  \bigoplus_{\substack{W\oplus^\perp \bigoplus^\perp
        V_j = X
        \\ \dim V_j = a_j -1 \\ \dim W = b}} H_{b+1}(S^\tw\smash F_\dotp^W) \otimes \bigotimes_{j=1}^i H_{a_j}
  (S^\tw \smash
  F_\dotp^{V_j}) \\ & \cong  H_{d+1}\bigg( \bigvee_{\substack{W\oplus^\perp \bigoplus^\perp
        V_j = X
        \\ \dim V_j = a_j -1 \\ \dim W = b}}S^\tw \smash F_\dotp^W \smash
    \bigwedge_{j=1}^i
    (S^\tw \smash
    F_\dotp^{V_j})\bigg) = H_{d+1} (S^\tw \smash J_\dotp^{\vec A}).
\end{align*}
Therefore $\alpha_{\vec A}$ could be produced by giving maps of simplicial sets
$S^\tw \smash J_\dotp^{\vec A} \rto S^\tw \smash F_\dotp^{\vec A}$ which give
isomorphisms on $H_{d+1}$, compatible with the images of arrows in $\I_n$.  These
arrows are given by Dehn invariants; unfortunately, while
Definition~\ref{def:Di} gives a geometrization of the Dehn invariant on
$F_\dotp^{\vec A}$, we do not have an analogous geometrization of the Dehn
invariant on $J_\dotp^{\vec A}$. Therefore the compatibility conditions between
the $\alpha_{\vec A}$ cannot be stated using only maps of simplicial sets.
Instead, \emph{ad-hoc} mappings are constructed between these spaces, which with
a scaling correction behave correctly on homology.  These maps are homotopy
equivalences after tensoring with $\Zinv$, although not \emph{integral} homotopy
equivalences.

\begin{remark}
  This seems to imply that the original definition of the Dehn invariant had an
  extra factor of $2$ somehow incorporated into the definition.  It would be
  interesting to see a geometric explanation of this phenomenon.
\end{remark}

To construct this explicit descriptions of $S^1$ and $S^\tw$ are required.  The
structure is summarized in the following table; note that in the case of $S^1$,
$\epsilon = 1$ always; in the case of $S^\tw$, $\epsilon = \pm 1$.
\begin{equation} \label{eq:circletable}
  \begin{tabular}{cl|cc}
    & & $S^1$ & $S^\tw$ \\
    \hline
    \multicolumn{2}{c|}{$n$-simplices} &$\{*,1,\ldots,n\}$ & $\{*,\circledast,\pm1,\ldots,\pm n\}$ \\
    $d_j(\epsilon i)$ & $j = 0, i = 1$ &  $*$ & $\circledast$ \\
    & $j = i = n$ & $*$ & $*$ \\
    & $j < i$ & $i-1$ & $\epsilon(i-1)$ \\
    & otherwise & $i$ & $\epsilon i$ \\
    $s_j(\epsilon i)$ & $j < i$ & $i+1$ & $\epsilon(i+1)$ \\
    & $j \geq i$ & $i$ & $\epsilon i$ \\
    \multicolumn{2}{c|}{$\Z/2$-action} & none & $\epsilon i \mapsto -\epsilon
                                                i$ 
  \end{tabular}
\end{equation}
All simplices above dimension $n$ are degenerate.  All face and degeneracy maps
on $*$ (resp. $\circledast$) map it to $*$ (resp. $\circledast$) in the appropriate
dimension.  

In this notation, the simplicial weak equivalence mentioned in
Lemma~\ref{lem:sma->join} is described via
\begin{equation} \label{eq:sma->join}
  (i,x,y)\in (S^1\smash X \smash Y)_n \rgoesto (d_i^{n-i+1}x, d_0^{i+1}y) \in (X
  \redjoin Y)_n.
\end{equation}

The key construction for the desired equivalence is the $\Z/2\times
\Z/2$-equivariant map defined by
\begin{align*}
  \gamma:\ &S^\tw \smash S^\tw \rto S^\tw \smash S^1 \\
  & (a,b) \rgoesto ((\sgn b)a,|b|)
\end{align*}
and $\gamma(\star) = *$.  Here, $\Z/2\times \Z/2$ acts on the left
coordinatewise, and on the right via the addition mapping
$\Z/2\times \Z/2\rto \Z/2$.  This is a \emph{two-fold} cover of $S^2$ by $S^2$.
More visually, consider the following illustration:
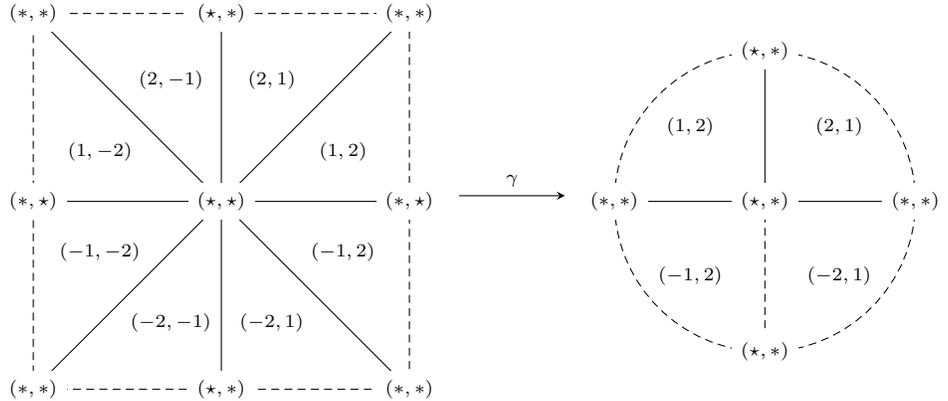
\[
  \begin{tikzpicture}[baseline, scale=2.5, font=\scriptsize]
    \coordinate (A1) at (1,1);
    \coordinate (A2) at (-1,1);
    \coordinate (A3) at (-1,-1);
    \coordinate (A4) at (1,-1);
    \draw[densely dashed] (-1,-1) rectangle (1,1); 
    \draw (0,1) -- (0,-1) (-1,0) -- (1,0);
    \draw (A1) -- (A3) (A2) -- (A4);

    \node[fill=white] at (A1) {$(*,*)$};
    \node[fill=white] at (A2) {$(*,*)$};
    \node[fill=white] at (A3) {$(*,*)$};
    \node[fill=white] at (A4) {$(*,*)$};
    \node[fill=white] at (1,0) {$(*,\star)$};
    \node[fill=white] at (-1,0) {$(*,\star)$};
    \node[fill=white] at (0,1) {$(\star,*)$};
    \node[fill=white] at (0,-1) {$(\star,*)$};
    \node[fill=white] at (0,0) {$(\star,\star)$};

    \node[fill=white] at (22.5:0.7) {$(1,2)$};
    \node[fill=white] at (67.5:0.7) {$(2,1)$};
    \node[fill=white] at (-22.5:0.7) {$(-1,2)$};
    \node[fill=white] at (-67.5:0.7) {$(-2,1)$};
    \node[fill=white] at (112.5:0.7) {$(2,-1)$};
    \node[fill=white] at (157.5:0.7) {$(1,-2)$};
    \node[fill=white] at (-157.5:0.7) {$(-1,-2)$};
    \node[fill=white] at (-112.5:0.7) {$(-2,-1)$};
  \end{tikzpicture}
  \ \setlen{4em}{\rto^\gamma}\ 
  \begin{tikzpicture}[baseline, scale =2, font=\scriptsize]
    \draw[densely dashed] (0,0) circle (1);
    \draw (0,1) -- (0,0) (-1,0) -- (1,0);
    \draw[densely dashed] (0,0) -- (0,-1);
    \node[fill=white] at (1,0) {$(*,*)$};
    \node[fill=white] at (-1,0) {$(*,*)$};
    \node[fill=white] at (0,1) {$(\star,*)$};
    \node[fill=white] at (0,-1) {$(\star,*)$};
    \node[fill=white] at (0,0) {$(\star,*)$};

    \node[fill=white] at (45:0.7) {$(2,1)$};
    \node[fill=white] at (135:0.7) {$(1,2)$};
    \node[fill=white] at (-135:0.7) {$(-1,2)$};
    \node[fill=white] at (-45:0.7) {$(-2,1)$};
  \end{tikzpicture}.
\]
In this diagram  all nondegenerate simplices present in
$S^\rho \times S^\rho$ and $S^\rho \times S^1$ are drawn; everything drawn dashed is
collapsed to a single point in the smash product.  Edges are not labeled but
$2$-simplices are, using the explicit description of the simplicial structures
in (\ref{eq:circletable}).  In effect, the map $\gamma$ is the endomorphism of
the unit disk which multiplies the angle (in polar coordinates) by $2$.

It is now possible to define
$f_{\vec A}: S^\tw \smash J^{\vec A}_\dotp \rto S^\tw \smash F^{\vec
  A}_\dotp$.  For any simplicial sets $K$ and $L$, let
$f: S^1\smash K \smash L \rto K \redjoin L$ take $(a, x, y)$ to
$(d_a^{n-a+1}x, d_0^{a+1}y)$; by Lemma~\ref{lem:smashToJoin} it is a weak
equivalence.  Define $f_{\vec A}$ inductively, as an $i$-fold composition of
maps of the following form:
\[S^\tw \smash K \smash S^\tw \smash L \rto^\tau S^\tw \smash S^\tw \smash K
  \smash L \rto^{\gamma} S^\tw \smash S^1 \smash K \smash L \rto^f S^\tw
  \smash K \redjoin L.\]

\begin{lemma} \label{lem:fratwe} After inverting $2$, $f_{\vec A}$ becomes an
  $I(X)$-equivariant weak equivalence.
\end{lemma}

With $\alpha_{\vec A}$ constructed, we can describe the homology groups of
$\YY(\vec A)_{hI(X)}$ explicitly in terms of group homology and scissors
congruence groups:
\begin{lemma} \label{lem:homologyComp}
  Let $\vec A = (b,a_1,\ldots,a_i)$ be an object of $\I_d$.  Then 
  \[ \tilde H_q\left(\YY(\vec A)_{hI(X)};\Q\right) \cong \bigoplus_{\ell_0 +
      \cdots + \ell_{i} = q} H_{\ell_0-b-1}(I(X^b), \P(X^b,1)_\Q) \otimes
    \bigotimes_{j=1}^{i} H_{\ell_j-a_j-1} (O(a_j+1), \P(S^{a_j},1)_\Q).\]
  Moreover, when $q = d+1$ this works with only $2$ inverted:
  \[H_{d+1}\left(\YY(\vec A)_{hI(X)};\Zinv\right) \cong \P(X^b) \otimes
    \bigotimes_{j=1}^{i} \P(S^{a_j}).\]
\end{lemma}

\begin{proof}
  Since $\alpha_{\vec A}$ is an equivariant weak equivalence, we know that
  \[ \tilde H_q\left(\YY(\vec A)_{hI(X)};\Zinv\right) \cong \tilde H_q\left(
      (S^\sigma \sma J_\dotp^{\vec A})_{hI(X)};\Zinv\right).\] We thus focus on
  the computation of the right-hand side.  We can model
  $(S^\sigma \sma J_\dotp^{\vec A})_{hI(X)}$ as the bisimplicial set
  $K_{\dotp,\dotp}$ where
  \[K_{p,q} = I(X)^q \sma (S^\sigma \sma J_\dotp^{\vec A}).\] Here, the
  horizontal face and degeneracies act only on the
  $ (S^\sigma \sma J_\dotp^{\vec A})$-coordinate.  The $j$-th vertical
  degeneracies add an identity into the $j$-th slot in the tuple $I(X)^q$; the
  vertical face maps are defined as follows:
  \[d_i((g_1,\ldots,g_q), y) =
    \begin{cases}
      ((g_1,\ldots,g_{i+1}g_i, \ldots,g_q),y) \caseif 0 < i < q \\
      ((g_1,\ldots,g_{q-1}),y) \caseif i = q \\
      ((g_2,\ldots, g_q), g_1\cdot y) \caseif i = 0.
    \end{cases}\] In particular, every simplicial set $K_{p,\dotp}$ is the nerve
  of a category whose objects are $Y_p$ and in which
  $\Hom(y,y') = \{g\in I(X)\,|\, g \cdot y = y'\}$ (plus a disjoint basepoint).

  Fix a single decomposition $X = W \oplus^\perp \bigoplus^\perp V_j$, and write
  $Y' = (S^\sigma \sma F_\dotp^W \sma \bigwedge_{j=1}^i (S^\sigma \sma
  F_\dotp^{V_j})$.  Consider the bisimplicial subset $K'_{\dotp,\dotp}$ of
  $K_{\dotp,\dotp}$ containing those simplices $((g_1,\ldots,g_q), y)$ with
  $y\in Y'$; for each $p$, $K'_{p,\dotp}$ is also the nerve of a category, which
  is the full subcategory of $K_{p,\dotp}$ on the simplices in $Y'$.  Note that
  the inclusion of this subcategory is essentially surjective, and thus induces
  a weak equivalence on nerves.  Thus the inclusion
  $K'_{p,\dotp} \rto K_{p,\dotp}$ is a weak equivalence for all $p$, and thus
  the inclusion $K'_{\dotp,\dotp} \rcofib K_{\dotp,\dotp}$ is a weak equivalence
  on geometric realization.  Consequently,
  \[\tilde H_q\left(\YY(\vec A)_{hI(X)};\Zinv\right) \cong \tilde H_q(K'_{\dotp,\dotp}).
  \]

  Any group element appearing in $K'_{\dotp,\dotp}$ must preserve each of
  $W,V_1,\ldots,V_i$.  Thus (by reversing the construction of $K_{\dotp,\dotp}$
  above) we see that
  \[K'_{\dotp,\dotp} \cong Y'_{h(I(W)\times \prod_{j=1}^i I(V_j))} \simeq
    (S^\sigma \sma F^W_\dotp)_{hI(W)} \sma \bigwedge_{j=1}^i (S^\sigma \sma
    F^{V_i}_\dotp)_{hI(V_i)}.
  \]
  After tensoring with $\Q$, the homology of the right-hand side gives the
  desired formula.  Moreover, as the $d+1$-st homology on the right is the
  lowest nonzero homology group, the second part of the lemma follows by the
  Kunneth theorem.
\end{proof}

It follows that the objects in the desired cube are isomorphic to the objects in
Goncharov's cube.  
In an ideal world, it would be possible to define
$\alpha_{\vec A} = H_{d+1}(f_{\vec A})$.  Unfortunately, it is not that simple,
as this is not compatible with Dehn invariants.  Consider a small
example.  When $\vec A = (d)$,
$F_\dotp^{\vec A} = J_\dotp^{\vec A} = F_\dotp^X$.  Fix $b$, and consider the
Dehn invariant $D_b$ corresponding to the morphism $(d) \rto (b,a)$.  This produces
the following \textbf{noncommutative} diagram:
\begin{diagram}[5em]
  { &  H_{d+1}(S^\tw \smash J_\dotp^{(b,a)}) \\
    H_{d+1}(S^\tw \smash F_\dotp^X) & H_{d+1}(S^\tw \smash F_\dotp^{(b,a)}) \\};
  \to{2-1}{1-2}^{\hat D_b} \to{2-1}{2-2}_{H_{d+1}(D_b)} \to{1-2}{2-2}^{H_{d+1}(f_{(b,a))}}
  \node[xshift=0.2em, yshift=0.8em] at (m-2-2.north west) {$\not\!\vcenter{\hbox{\rotatebox{-90}{$\circlearrowright$}}}$};
\end{diagram}
To make the diagram commute it is necessary to multiply the vertical map by
$\frac12$ (due to $\gamma$ having degree $2$).  This is true in general; if
$|\vec A| > 1$ the construction of $f_{\vec A}$ contains $|\vec A|-1$
compositions with $\gamma$, and thus multiplies by $2^{|\vec A|-1}$ in homology.
As $2$ is inverted, this can be remedied:


\begin{lemma}
  $\alpha_{\vec A} \defeq 2^{1-|\vec A|}H_{d+1}(f_{\vec A})$ gives a natural isomorphism
  $\DD^\eqvt \Rto H_{d+1}(\YY;\Zinv)$.
\end{lemma}

The proof is now complete. \qed

\section{Large cubes and the Dehn complex}
\label{sec:large}

To construct the Dehn complex, Goncharov essentially starts with the groups
$\P(X,1)$, takes their coinvariants with respect to a group action (or, in other
words, a homology group), constructs a new differential to make these groups
into a chain complex, and then studies the homology of this new chain complex.
This produces an object which is difficult to analyze and does not seem to fit
into any of the standard methods for taking the homology of homologies.  In
light of Example~\ref{example:cube}, this chain complex can be thought of as the
total homotopy cofiber of a cube of groups.  The goal of this section is to
compare it with the total homotopy cofiber $(\sY^X)_{hI(X)}$ of derived Dehn
invariants defined above.  It turns out that the homology of $(\sY^X)_{hI(X)}$
can be described in two ways: one by directly analyzing its homotopy type, and
one via a spectral sequence of Munson and Volic \cite[Proposition
9.6.14]{munsonvolic}.  The comparison between these computations is incredibly
fruitful.

For the rest of this section, the neat geometry $X$ over $k$ has dimension
$d = 2n$ or $2n-1$ (so that $n = \lfloor \frac{d-1}{2} \rfloor$).

\begin{theorem} \label{thm:reallycool} After inverting $2$ there is an
  equivalence
  \[(\sY^X)_{hI(X)} \simeq (S^\tw \smash S^{n-1})_{hI(X)}.\] Here $I(X)$
  acts by $\det$ on the $S^\tw$-coordinate and trivially on the
  $S^{n-1}$-coordinate.
\end{theorem}

The proof of this theorem is deferred to Section~\ref{sec:twoproofs}.  The key
to this theorem is the observation that homotopy coinvariants are both homotopy
colimits and therefore \emph{commute}; thus
\[(\sY^X)_{hI(X)} = (\tcofib \YY)_{hI(X)} \simeq \tcofib (\YY_{hI(X)}) .\] This
means that the simple combinatorial nature of $\YY$ can be played against the
benefits of taking homotopy coinvariants. This is also where the benefits of
constructing an \emph{equivariant} Dehn invariant comes into play: if an
equivariant model for the Dehn invariant did not exist it would be impossible to
move the homotopy coinvariants outside of the total cofiber.

\begin{theorem} \label{thm:identGonch}
  Write $\Zinv^\tw$ for a copy of $\Zinv$ with $I(X)$ acting on it via multiplication by
  the determinant. For all $i$,
  \[\tilde H_i\left(\sY^X_{hI(X)};\Zinv\right) \cong H_{i-n}\left(I(X); \Zinv^\tw\right);\]
  in particular, when $i<n$ both are zero.
\end{theorem}

\begin{proof}
  By Theorem~\ref{thm:reallycool} there is an isomorphism
  $H_i(\sY^X_{hI(X)}; \Zinv) \cong H_i((S^\tw \smash S^{n-1})_{hI(X)}; \Zinv)$.
  By the homotopy orbit spectral sequence (Proposition~\ref{prop:hoss}), and
  since $\Zinv^\tw \cong \tilde H_1(S^\tw;\Zinv)$,
  \[\tilde H_i\left((S^\tw \smash S^{n-1})_{hI(X)};\Zinv\right) \cong H_{i-n}\left(I(X);
    \tilde H_n\left(S^\tw\smash S^{n-1};\Zinv\right)\right) \cong H_{i-n}\left(I(X);\Zinv^\tw\right).\]
\end{proof}

The spectral sequence for the total homotopy colimit of a cube proved in
Proposition~\ref{prop:sstcofib} can now be used to to connect the homotopy type
of $\sY^X_{hI(X)}$ to the Dehn complex.  In this case the spectral sequence becomes
\begin{equation} \label{eq:gonchSS}
  E^1_{p,q} = \bigoplus_{\vec A = (b,a_1,\ldots,a_{n-1-p})} \tilde
  H_q\left(\YY(\vec A)_{hI(X)};\Zinv\right) \Rto \tilde H_{p+q}\left(\sY^X_{hI(X)};\Zinv\right).
\end{equation}
Since $\YY(\vec A)$ has no nonzero homology below degree $d+1$ this also holds
for $\YY(\vec A)_{hI(X)}$.
Thus all entries in the spectral sequence with $q < d+1$ are $0$.  When
$q = d+1$ the row of the spectral sequence is exactly the Dehn complex.

Below is a picture of the spectral sequence for
$\tilde H_\ast (\sY^X_{I(X)};\Zinv)$ in (\ref{eq:gonchSS}). The red indicates
the non-zero entries in $E^1$. The Dehn complex is the base complex of
$\mathbf{Y}$; it is the thick blue line sitting in the row where $q=d+1$. 

\begin{equation} \label{diag:ss} 
\begin{tikzpicture}[yshift=-5em, baseline]
   \node[below] at (3,0) {};
   \node[left] at (0,3) {$d+1+n$};
   \node[below] at (2,0) {$d+1$};
   \node[left] at (0,2) {$d+1$};

   \draw (2, 0) -- (0, 2);
   \draw (3, 0) -- (0, 3);
   \fill[blue] (0, 2) circle (2pt);
   \fill[blue] (1,2) circle (2pt);
   \draw[blue, very thick] (0,2) -- (1, 2);
   \draw[red] (0, 2) -- (0, 4);
   \draw[red] (1, 2) -- (1, 4);
   \fill [red, opacity=.2] (0,2) rectangle (1,4);
   \draw [<->,thick] (0,4) node (yaxis) [above] {$q$}
   |- (4,0) node (xaxis) [right] {$p$};

   \node[right] at (2.5,2.5) {$\displaystyle{E^1_{p,q} =
       \bigoplus_{\substack{\vec A \in \I_d\\ \mathrm{len}(\vec A)=n-p}} \tilde
       H_q\left(\YY(\vec A)_{hI(X)};\Zinv\right) \Rto \tilde
       H_{p+q-n}\left(I(X);\Zinv^\tw\right)}.$};
   \node[right] at (6,1.5) {$d_1:E^1_{p,q} \to E^1_{p-1,q}$};
   \node[right] at (6,1) {$d_r:E^r_{p,q} \to E^r_{p-r,q-1+r}$};
   \node[right] at (6,0.5) {$\mathrm{len}(b,a_1,\ldots,a_i) = i+1$};
 \end{tikzpicture}
\end{equation}

\begin{definition} \label{def:epsilon} Let $F^*_{p,q}$ be a first-quadrant
  homologically-graded spectral sequence with lowest nonzero row at $p=m$,
  converging to the sequence of groups $G_{p+q}$.  The \emph{base complex} of
  the spectral sequence is the complex $F^1_{*,m}$ with differential $d^1$.  The
  induced homomorphism $\theta_n:G_n \rto F^1_{n-m,m}$ is called the \emph{projection to
    the base}.\footnote{This is also sometimes called an ``edge homomorphism in
    the spectral sequence.''}
\end{definition}

\begin{theorem} \label{thm:projbase}
  Projection to the base gives a homomorphism
  \[\theta_m: H_{n+m}(I(X); \Zinv^\sigma) \rto H_m\P_*(X).\]
  This homomorphism is surjective if and only if all differentials
  $d^r: E^r_{m,d+1} \rto E^r_{m-r,d+r}$ for $r \geq 2$ are zero.  It is
  injective if and only if $E^\infty_{p,m-p} = 0$ for all $p > d+1$.
\end{theorem}

\begin{proof}
  Projection to the base is surjective onto $E^\infty_{m,d+1}$.  As the base
  complex is the lowest nonzero row, the spectral sequence contains no
  differentials \emph{into} that row; thus, $E^\infty_{p,d+1}$ is a subgroup of
  $E^2_{p,d+1}$ for all $p$.  This subgroup is exactly the intersection of all
  of the kernels of the $d^r$, and is therefore equal to the whole group if and
  only if all of these differentials are $0$.  The kernel of projection to the
  base is exactly the subgroup given by all of the terms in the $m$-th diagonal
  above the base complex on the $E^\infty$-page; thus the map is injective if
  and only  the terms in the diagonal are $0$.
\end{proof}

In the case of this spectral sequence, it is possible to give an explicit
description of $\theta_{n-1}$:
\begin{lemma} \label{lem:edge-desc} Let $k = \R$.  In the spectral sequence
  (\ref{eq:gonchSS}), the map
  \[\theta_{n-1}: H_{d}(I(X);\Zinv^\sigma) \to H_{n-1}\P_*(X)\]
  is induced by the map taking a chain $(g_1,\ldots,g_{d})$ to
  the scissors congruence class of the $d$-simplex with vertices
  \[\big\{ x_0, g_{d}x_0, g_{d}g_{d-1}x_0,\ldots,g_{d}\cdots g_1 x_0
    \},\] (for any chosen point $x_0\in X$) with the sign given by
  $\prod_{i=1}^{d} \det(g_i)$.
\end{lemma}
The proof of this lemma is technical and not illuminating, so it is postponed to
Section~\ref{sec:twoproofs} (Lemma~\ref{lem:edge-desc-k}); in fact, in that
section it is proved over any field $k$.  We state it here for the special case
as it makes the result easier to describe and this is the only case of interest
in the current paper.

Directly from the spectral sequence it is possible to prove minor
generalizations (to all fields instead of algebraically closed ones, and
including all $d$ not just odd ones, and removing rational coefficients) of the
following results of Cathelineau:
\begin{theorem}[{Generalization of \cite[Thm. 10.1.1]{cathelineau_bar}}] \label{thm:cathelineau}
  For $d=2n$ or $2n-1$ with $n\geq 2$, and any field $k$ of characteristic $0$,
  \[H_i(I(X);\Zinv^\sigma) = 0 \qquad \hbox{if }i < n\]
  and
  \[H_{n}(I(X);\Zinv^\sigma) \cong H_0\P_*(X^d).\]
\end{theorem}
\begin{proof} Consider the spectral sequence (\ref{diag:ss}).  Everything below
  the $d+1$-st diagonal is $0$, which implies the first claim.  In the $p+q = d+1$-st
  diagonal there is exactly one nonzero entry: $H_0(\P_*(X^d))$.
\end{proof}

For the next theorem it is unfortunately necessary to rationalize, rather than
simply inverting $2$.  As the rationalization is used only in computing the
group homology of $SO(2)$, in certain cases with good control over the torsion
in this group it may be possible to get away with a milder localization.

\begin{theorem}[{Generalization of \cite[Proposition 6.2.2]{cathelineau_proj}}]
  Let $X$ be a neat geometry over a field $k$ of characteristic $0$.  Then
  \[H_1(I(X^1), H_1(F^{X^1}_\dotp)^\sigma) = 0.\] Consequently, for $d = 2n$ or $2n-1$, with
  $n \geq 2$,
  \[H_{n+1}(I(X);\Q^\sigma) \cong H_1(\P_*(X^d)_\Q).\]
\end{theorem}

\begin{proof}
  Via the spectral sequence in (\ref{diag:ss}), it suffices to check that
  $E^1_{0,d+2} = 0$.  This term is a direct sum of tensor products, where each
  tensor product contains exactly one of the groups
  \[H_1(I(\cH^1), H_1(F^{\cH^1}_\dotp)^\sigma) \qquad\hbox{or}\qquad H_1(I(S^1),
    (F^{S^1}_\dotp)^\sigma).\] When $X$ is spherical, only terms of the second
  sort will arise.  Thus the second statement in the theorem follows directly
  from the first.

  Let $I^+(X^1)$ be the subgroup of $I(X^1)$ of those elements with determinant
  $1$.  The Lyndon--Hochschild--Serre spectral sequence for the extension
  \[I^+(X) \rto I(X) \rto \Z/2.\]
  has
  \[E^2_{p,q} = H_p(\Z/2, H_q(I^+(X), H_1(F^{X^1}_\dotp)^\sigma)) \Rto
    H_{p+q}(I(X), H_1(F^{X^1}_\dotp)^\sigma).\] In particular, since $H_1(F^{X^1}_\dotp)$
  is $2$-divisible, $E^1_{p,q} = 0$ whenever $p > 0$.  Thus
  \begin{equation} \label{eq:Sttogp} H_n(I(X),H_1(F^{X^1}_\dotp)^\sigma) \cong
    H_0(\Z/2, H_n(I^+(X^1), H_1(F^{X^1}_\dotp)^\sigma)).
  \end{equation}
  
  We turn our focus to computing
  $H_n(I^+(X^1), H_1(F^{X^1}_\dotp)^\sigma) = H_n(I^+(X^1),
  H_1(F^{X^1}_\dotp))$, which by the homotopy orbit spectral sequence
  (Proposition~\ref{prop:hoss}) is isomorphic to
  $H_{n+1}((F^{X^1}_\dotp)_{I^+(X^1)}; \Q)$.  Consider
  $(F^{X^1}_\dotp)_{hI(X^1)}$ as a bisimplicial set, with the simplices of
  $F^{X^1}_\dotp$ in the horizontal direction, and the $I(X^1)$-action in the
  vertical direction.  This produces a double complex, whose spectral sequence
  converges to $H_*((F^{X^1}_\dotp)_{I^+(X^1)}; \Q)$.

  Consider the spectral sequence in which we take first vertical homology, then
  horizontal homology.  As the only nondegenerate simplices in $F_\dotp^{X^1}$
  are in dimensions $0$ and $1$, this spectral sequence will be concentrated in
  the first two columns.  The simplicial set $F_\dotp^{X^1}$ has exactly one
  non-basepoint $0$-simplex $[X^1]$, which has as its stabilizer $I^+(X^1)$;
  thus
  \[E^1_{0,q} \cong H_q(I^+(X^1),\Q).\] The simplicial set $F_\dotp^{X^1}$
  has as its nondegenerate $1$-simplices inclusions $U_0 \subseteq X^1$, which
  have as their stabilizers the subgroup of $I(X^0) \times I(X^0)$ with
  determinant $1$ (this is the group which fixes $U_0$ and $U_0^\perp$).  This
  is a $2$-group, and we will therefore have
  \[E^1_{1,q} \cong 0 \qquad q > 0.\] When $q = 0$ this is simply $\Q$, and
  $d_1$ will send this $\Q$ isomorphically to $E^1_{0,0} \cong \Q$.  Thus the
  spectral sequence collapses at $E^2$, where it is concentrated in the $0$-th
  column, producing
  \[H_q(I^+(X^1),H_1(F^{X^1}_\dotp)^\sigma) \cong H_{q+1}(I^+(X^1),\Q^\sigma).\]
  Combining this with (\ref{eq:Sttogp}) produces
  \[H_n(I(X^1), H_1(F^{X^1}_\dotp)^\sigma) \cong H_0\left(\Z/2,
      H_{n+1}\left(I^+(X^1),\Q^\sigma\right)\right).\] In the particular case of
  interest we have
  \[H_1(I(X^1), H_1(F^{X^1}_\dotp)^\sigma) \cong H_0\left(\Z/2, H_2\left(I^+(X^1),
        \Q^\sigma\right)\right).\]
  The group $I^+(X^1)$ is abelian, and therefore
  \[H_2\left(I^+(X^1), \Q^\sigma\right) \cong \Lambda^2 (I^+(X^1)\otimes \Q^\sigma).\]
  The group $\Z/2$ acts by $-1$ on both $\Q$ and on $I^+(X^1)$; thus it will act
  by $-1$ on every chain, and thus on the homology group $H_2$.  Thus, since
  everything is $2$-divisible, 
  \[H_0\left(\Z/2, H_2\left(I^+(X^1),\Q^\sigma\right)\right) \cong 0,\]
  as desired.
\end{proof}

Directly from these calculations we can conclude the following:

\begin{corollary}
  Projection to the base $\theta_m$ is an isomorphism when $m = 0,1$ and
  surjective when $m = 2$.
\end{corollary}

\section{Goncharov's conjectures} \label{sec:conj}

In this section we discuss the connections between Goncharov's original
conjectures, Cheeger--Chern--Simons invariants, and the results of the previous
sections.

\subsection{Projection to the base and the modified conjectures}

In \cite{goncharov}, Goncharov has a series of three conjectures about possible
connections between the Dehn complex and the algebraic $K$-theory of $\CC$.  We
give a summary of these conjectures here.  Our notation does not exactly agree
with Goncharov's; in particular, Goncharov's Dehn complex is cohomologically
graded and $1$-indexed, while ours is homologically graded and $0$-indexed.  We
number the parts of our summary by the number of the conjecture in
\cite{goncharov}.

All tensor products of $\Z/2$-modules in this section are equipped with a
$\Z/2$-action via the diagonal action.  Write
$\Q^{n\tw} \defeq (\Q^\tw)^{\otimes n}$, equipped with the diagonal action
of $\Z/2$.

\begin{conjecture}[{\cite[Conjectures 1.7-1.9]{goncharov}}]
  Let $\P_\ast(X^{2n-1})$ be the Dehn complex for the geometry $X^{2n-1}$ over
  $\R$.
  \begin{itemize}
  \item[(1.8)] There exist homomorphisms
    \[H_i\P_\ast(S^{2n-1}) \rto^{\phi_i} (\gr^\gamma_nK_{n+i}(\CC)_\Q\otimes Q^{n\tw})^+\]
    and
    \[H_i\P_\ast(\cH^{2n-1}) \rto^{\phi_i} (\gr^\gamma_nK_{n+i}(\CC)_\Q\otimes
      \Q^{n\tw})^-.\]
  \item[(1.7)] The homomorphism $\phi_{n-1}$ is injective, and the diagrams
    \begin{diagram}
      { \ker \bigoplus_{i=1}^{n-1} \hat D_i \cong \hspace{-3em}& H_{n-1}\P_*(S^{2n-1}) & & (\gr_n^\gamma K_{2n-1}(\CC)_\Q \otimes
        \Q^{n\tw})^+ \\
        & & \R/(2\pi)^n\Z \\};
      \to{1-2}{1-4}^{\phi_{n-1}}
      \to{1-2}{2-3}_{\mathrm{vol}} \to{1-4}{2-3}^r
    \end{diagram}
    and
    \begin{diagram}
      { \ker \bigoplus_{i=1}^{n-1} \hat D_i \cong \hspace{-3em} &
        H_{n-1}\P_*(\cH^{2n-1}) & & (\gr_n^\gamma K_{2n-1}(\CC)_\Q \otimes
        \Q^{n\tw})^- \\
        & & \R \\};
      \to{1-2}{1-4}^{\phi_{n-1}} \to{1-2}{2-3}_{\mathrm{vol}}
      \to{1-4}{2-3}^r
    \end{diagram}
    commute. 
    Here, the right-hand map is the Borel (resp. Beilinson) regulator.
  \item[(1.9)] All $\phi_i$ are  isomorphisms.
  \end{itemize}
  Here, $\gr^\gamma_n$ is the $n$-th graded part of the $\gamma$-filtration, and
  $\Q^{n\tw}$ is the vector space $\Q$ with $\Z/2$ acting on it via
  multiplication by $(-1)^n$.  The sign in the superscript indicates taking the
  $\pm 1$ eigenspace with respect to the action by complex conjugation.
\end{conjecture}
For an exposition of the $\gamma$-filtration, see for example
\cite{grayson_filtration}.  For an exposition of the Borel and Beilinson
regulators see \cite[Chapter 9]{burgosgil}.

Goncharov proves (1.7) in the case when $\CC$ is replaced with $\bar \Q$ and
simplices in the Dehn complex are restricted to those with algebraic vertices
\cite[Theorem 1.6]{goncharov}.  Note that any polytope which can appear as the
fundamental domain of a group action is automatically in the kernel of all Dehn
invariants; thus in particular Goncharov's conjectures would imply that all
volumes of hyperbolic manifolds must be in the image of the Borel regulator.

Inspired by the conjectures, we propose an alternative method to connect the
algebraic $K$-theory of $\CC$ and the scissors congruence groups (see
Proposition~\ref{cor:zigzags}).  Explaining the $\gamma$-filtration and the Borel
and Beilinson regulators in the above theorems is extremely nontrivial, while
the corresponding notions in our approach are much more elementary.  Due to the
nature of our methods the morphisms we construct go in the opposite direction to
Goncharov's desired morphisms.

As a first observation, Theorem~\ref{thm:projbase} constructs exactly the
morphism desired by Goncharov's Conjecture 1.8.  In the rest of this section we
will analyze how well this morphism proves Conjectures 1.7 and 1.9.

We begin with a description of the Cheeger--Chern--Simons class, which plays the
same role as the Borel regulator for the case of orthogonal groups (rather than
general linear groups).  

The construction is originally due to Cheeger
and Simons \cite[Section 8]{cheegersimons} (although the authors originally
learned it from Dupont \cite[Sect. 10]{dupont}) and works for more general
homogeneous spaces. See also \cite{dupont_kamber} and \cite[Section
5]{dupont_hain_zucker}.

\begin{definition}
  The \emph{Cheeger--Chern--Simons construction} is a homomorphism
  \[\mathscr{CCS}: H_{2n-1}(O(2n;\R);\Zinv^\tw) \to \widehat\P(S^{2n-1})/[S^{2n-1}],\] defined as follows.
  Consider the \emph{space} $O(2n;\R)/O(2n-1;\R)$ with the usual topology.  This
  is homeomorphic to $S^{2n-1}$ with a distinguished point.  The group
  $O(2n;\R)$ acts on this on the left, moving the distinguished point.  A chain
  in degree $2n-1$ is represented by a sequence of elements
  $(g_1,\ldots,g_{2n-1})$; call such a chain \textsl{generic} if the points
  $\{x,g_{2n-1}x,g_{2n-1}g_{2n-2}x, \ldots, g_{2n-1}\cdots g_1x\}$ all lie in a
  single open hemisphere.  For a generic chain, define a geodesic simplex in
  $S^{2n-1}$ associated to this chain by
  \[\Delta_{(g_1,\ldots,g_{2n-1})} \defeq
    (x_0,g_{2n-1}x_0,g_{2n-1}g_{2n-2}x_0,\ldots,g_{2n-1}\cdots g_1x_0).\] To
  check that this morphism is well-defined it suffices to check that given any
  $2n$-chain $(g_1,\ldots,g_{2n})$ all of whose faces are generic, the sum over
  the boundary is $0$.  This holds in $\P(S^{2n-1})/[S^{2n-1}]$, and this
  construction can be extended to all of $H_{2n-1}(O(2n;\R);\Zinv^\tw)$ as the
  generic chains are dense.  Define
  \[\begin{array}{rccl}
    \mathscr{CCS}: & H_{2n-1}(O(2n;\R)^\delta;\Zinv^\tw) &\to& \widehat\P(S^{2n-1})/[S^{2n-1}] \\
    &[(g_1,\ldots,g_{2n-1})] &\rgoesto& [\Delta_{(g_1,\ldots,g_{2n-1})}] 
    \end{array}\]
  (See \cite[Section 8]{cheegersimons} and \cite{dupont_kamber} for a more
  in-depth discussion.)

  Now fix a volume form $v\in \Omega^{2n-1}(S^{2n-1})$ (which we normalize to so
  that $\int_{S^{2n-1}} v_{S^{2n-1}} = (2\pi)^n$).  The
  \emph{Cheeger--Chern--Simons class} is the homomorphism
  \[\CCS: H_{2n-1}(O(2n;\R);\Zinv^\tw) \mathrel{\setlen{5em}{\to^{\mathscr{CCS}}}}
    \widehat\P(S^{2n-1})/[S^{2n-1}] \to^\vol \R/(2\pi)^n\Zinv,\]

  Here we restricted to the even-dimensional case because (by the
  ``center-kills'' lemma \cite[Lemma 5.4]{dupont}) the homology groups
  $H_{2n-1}(O(2n;\R));\Zinv^\sigma)$ are all $0$.  However, this construction
  works equally well for odd-dimensional groups when this is not the case.  By
  considering $\mathcal{H}^d \cong O^+(1,d;\R)/O(d;\R)$ one obtains
  an analogous homomorphism
  \[\mathscr{CCS}: H_{d}(O(1,d);\Zinv^\tw) \to \P(\cH^{d})\]
  and 
  \[\CCS: H_{d}(O(1,d;\R)^\delta;\Zinv^\tw) \rto \R.\]
\end{definition}

A close analysis of the Cheeger--Chern--Simons construction gives the following
analog of Goncharov's Conjecture 1.7 for hyperbolic geometry:
\begin{theorem} \label{thm:hyperbolicGonch1.7} When $X$ is hyperbolic with $d = 2n$ or
  $2n-1$, projection to the base (Definition~\ref{def:epsilon})
  $ H_{2n-1}(O(2n,1;\R)^\delta,\Zinv^\tw) \to H_{n-1}\P_*(\cH_\R^{2n-1})$ fits
  into a commutative diagram
  \[
    \begin{tikzcd}
      H_{2n-1} (O(2n-1,1;\R)^\delta, \Zinv^\tw) \ar{rr}{\theta_{n-1}} \ar{dr}[swap]{\CCS} & & H_{n-1}\P_*(\cH_\R^{2n-1})  \ar{dl}{\vol} \\
      & \R. &
    \end{tikzcd}
  \]
  In particular, $\vol$ is injective if $\theta_{n-1}$ is surjective and $\CCS$ is injective.
\end{theorem}

\begin{proof}
  First, consider the hyperbolic case.  The key observation to prove the theorem
  is that the group homomorphism $\mathscr{CCS}$ agrees with the explicit
  description of $\theta_{n-1}$ in Lemma~\ref{lem:edge-desc} after composition
  with the isomorphism
  $\P(\cH^{d}) \rto H_{d+1}((S^\sigma \smash
  F^{\cH^{d}}_\dotp)_{I(\cH^{d})};\Zinv)$ in Theorem~\ref{thm:f<>p}.
\end{proof}

The main difficulty in relating the spherical Cheeger--Chern--Simons
homomorphism to volume is that (as mentioned in Remark~\ref{rem:spherVol})
volume is not well-defined on $\P(S^{2n-1})$.  An interesting question is whether
it is possible to give a well-defined definition of volume of
$H_{n-1}\P_*(S^{2n-1})$---i.e., on the kernel of the Dehn invariant.  However,
it may be the case that this group is nonzero, and yet still a well-defined lift
of the volume is possible, in which case an analogous statement to the
hyperbolic one should exist.  The precise relationship between the
Cheeger--Chern--Simons class and projection to the base in the spherical case is
the following; we omit the proof as it is directly analogous to the hyperbolic
case.

Let $p:\hat P(S^{2n-1}) \rfib \P(S^{2n-1})$ be the projection.  Let
\[L_{2n-1} \defeq p^{-1}(H_{n-1}\P_*(S^{2n-1}));\] this is the subgroup of
spherical polytopes with Dehn invariant equal to $0$ after reduction by lunes.

\begin{theorem} \label{thm:sphericalGonch1.7}
  In the spherical case, the Cheeger--Chern--Simons class factors through the
  inclusion
  \[L_{2n-1}/[S^{2n-1}] \rcofib \hat \P(S^{2n-1})/[S^{2n-1}]\] and is related to
  the projection to the base via the following commutative diagram:
  \begin{diagram}
    { H_{2n-1}(O(2n;\R);\Zinv^\tw) & L_{2n-1}/[S^{2n-1}] &
      \R/(2\pi)^n\Zinv \\
      & H_{n-1}\P_*(S^{2n-1}) \\};
    \to{1-1}{1-2}^{\mathscr{CCS}}
    \to{1-2}{1-3}^\vol \to{1-1}{2-2}_{\theta_{n-1}}
    \fib{1-2}{2-2}^{p}
    \node (m-100-100) at (m-1-3) {\phantom{$\R/(2\pi)^n\Zinv$}};
    \diagArrow{out=15, in=165, ->}{1-1}{100-100}^{\CCS}
  \end{diagram} 
\end{theorem}

Moreover, the analysis in Theorem~\ref{thm:cathelineau} produces the following
version of Goncharov's conjecture 1.8, with special cases of 1.9:

\begin{theorem}\label{thm:gonch1.8}
  Let $d = 2n$ or $2n-1$. For a neat geometry $X$ of dimension $d$,
  Projection to the base gives a homomorphism
  \[\theta_m: H_{n+m}(I(X), \Q^\sigma) \rto H_m\P_*(X)_\Q.\]
  This homomorphism is an isomorphism when $m = 0$ or $1$ and is surjective when
  $m = 2$.
\end{theorem}

\subsection{Relationship to algebraic $K$-theory}

In order to relate the homology of isometry groups to algebraic $K$-theory, we
introduce the rank filtration.

\begin{definition}[{\cite[p. 296]{weibel_kbook}}]
  The \emph{higher rational $K$-theory of $k$} is defined to be
  \[K_*(k)_\Q \defeq \hbox{primitive elements of the Hopf algebra } H_*GL(k).\]
  (Recall that all homology is taken with rational coefficients.) The rank
  filtration on $K_*(k)_\Q$ is defined by
  \[F_iK_*(k)_\Q \defeq  K_*(k)_\Q \cap \big(\im (H_* GL(i;k) \rto H_* GL(k))\big).\]
  Then
  \[\gr^\rk_nK_*(k)_\Q \defeq F_nK_*(k)_\Q/F_{n-1}K_*(k)_\Q.\]
  We will also need an auxiliary object; define
  \[CL_{n,m}(k) \defeq \coker \big( (H_m GL(n-1;k))_P \rto (H_m GL(n;k))_P\big),\]
  where $(H_m GL(i;k))_P$ is the subgroup of $H_m GL(i;k)$ of those elements
  whose images in $H_* GL(k)$ is primitive. Observe that there is a natural
  surjection $CL_{n,m}(k) \rfib \gr^\rk_n K_m(k)_\Q$.

  Analogously to $CL_{n,m}$ we define
  \[CO_{n,m}(k) \defeq \coker ((H_mSO(2n-2;k))_P \to H_mSO(2n;k)_P),\] where
  $(H_mSO(i;k))_P$ denotes those elements whose image in $H_*SO(k)$ is
  primitive.  We assume that the stabilization map $SO(2n-2;k) \to SO(2n;k)$
  adds coordinates at positions $n$ and $2n$, rather than at $2n-1$ and $2n$.
\end{definition}

There is an action of $\Z/2$ on $H_*SO(n;k)$ given by conjugation by a matrix
with determinant $-1$.  From the exact sequence
\[SO(n;k) \rto O(n;k) \rto \Z/2\] the homology of $SO(n;k)$ splits into
eigenspaces
\[H_mSO(n;k) \cong H_mO(n;k) \oplus H_m(O(n;k),\Q^\tw),\] where the first
component is the $+1$-eigenspace and the second is the $-1$-eigenspace
\cite[p.489]{cathelineau_homology_stability}.  Since stabilization is
equivariant with respect to this action, it induces an action on both
$\gr^\rk_nHK_*(k)_\Q$ and $CO_{n,m}(k)$.  When $n$ is odd
$H_m(O(n;k);\Q^\tw) = 0$ for all $m$ (\cite[Theorem
1.4]{cathelineau_homology_stability}, or \cite[Lemma 5.4]{dupont}); in
particular, this implies that
\begin{equation} \label{eq:CO<->H}
  CO_{n,m}(k)^- \cong H_m(O(2n;k),\Q^\tw).
\end{equation}

We now turn our attention to explaining the connection between rational homology
of orthogonal groups and algebraic $K$-theory.  The comparison between the two
is very well-studied (see, for example,
\cite{berrick_karoubi_schlichting_ostvaer}), and it is known that, after
rationalizing, Hermitian $K$-theory is isomorphic to the homotopy fixed points
of algebraic $K$-theory under the action sending a matrix to its transpose
\cite[(1-c)]{berrick_karoubi_schlichting_ostvaer} via the hyperbolic map
(defined below).  This is explored in more detail below, and used to explain the
connection between Goncharov's conjectures and the theorems proved above.  In
particular, we will explain why, in the real case, both spherical and hyperbolic
geometries arise, and how Goncharov's curious twisting factors $\Q^{n\tw}$
arise.
\begin{definition}
  The \emph{hyperbolic map}:
  \[\hyp:GL(n;k) \rto SO(n,n;k): \  M
    \rgoesto \begin{matrix}{cc} M \\ & (M^\tw)^{-1}\end{matrix}.\] When $k$
  contains $i = \sqrt{-1}$ and $\sqrt{2}$, conjugation by the matrix
  \[D_n \defeq \frac1{\sqrt 2}
    \left(\begin{array}{cc} I_n & I_n \\ -iI_n & iI_n \end{array}\right) \qquad
    \det D_n = i^n\]
  induces an isomorphism $SO(n,n;k) \cong SO(2n;k)$.  Conjugation by the matrix
  \[D^{[1]}_n \defeq \frac1{\sqrt 2} \left(
      \begin{array}{cc} I_n & I_n \\
        -\diag(1,i,\ldots,i) & \diag(1,i,\ldots,i) \end{array}\right) \qquad
    \det D^{[1]}_n = i^{n-1}\] induces an isomorphism
  $SO(n,n;k) \cong SO(1,2n-1;k)$.
\end{definition}

This map is not an isomorphism on homology, and it is known that in the limit as
$n \to \infty$ it has a large kernel.

\begin{definition}
  For any field extension $L/k$, the groups $H_mGL(n;k)$ and $H_mSO(2n;k)$ have
  an induced action by $\Gal(L/k)$.  We call this the \textsl{Galois action} and
  write it on the \textsl{left}.  We call the action induced by conjugation by a
  matrix with determinant $-1$ on $H_*SO(n;k)$ the \textsl{conjugation action},
  and we write it on the \textsl{right}.  These two actions commute, and are
  equivariant with respect to the stabilization maps, so induce actions on
  $CL_{n,m}$ and $CO_{n,m}$.  In the current context only quadratic extensions
  are considered, and $\pm1$-eigenspaces are denoted with a $+$ or a $-$, as
  before.  Thus, for example, the space $\tensor[^+]{CO_{n,m}(k)}{^-}$ is the
  subspace of those vectors which are $-1$-eigenvectors for the conjugation
  action and $+1$-eigenvectors for the Galois action.

  The module $\Q^\tw$ is considered to have ``Galois action'' by the sign
  action, and abuse notation to consider the ``Galois action'' on
  $CL_{m,n}(L)\otimes \Q^{n\tw}$ via the diagonal action.
\end{definition}

In the case when a field $k$ does not contain $\sqrt{-1}$ it is possible to
compare the scissors congruence of $k$ and the algebraic $K$-theory of $k(i)$.
The case when $k = \R$ is the case of Goncharov's original conjectures.
\begin{proposition} \label{cor:zigzags}
  Let $k$ be a field containing $\sqrt{2}$ and not containing $i = \sqrt{-1}$.
  There exist natural zigzags
  \[\tensor[^+]{(\gr_n K_m(k(i))_\Q\otimes \Q^{n\tw})}{} \lfib \tensor[^+]{(CL_{n,m}(k(i))\otimes
    \Q^{n\tw})}{} \rto \tensor[^+]{H_{m-n}\P_*(S^{2n-1}_{k(i)})}{} \lto
    H_{m-n}\P_*(S_k^{2n-1})\] and
  \[\tensor[^-]{(\gr_n K_m(k(i))_\Q\otimes \Q^{n\tw})}{} \lfib \tensor[^-]{(CL_{n,m}(k(i))\otimes
    \Q^{n\tw})}{} \rto \tensor[^+]{H_{m-n}(\P_*(\cH^{2n-1}_{k(i)}))}{}
    \lto H_{m-n}\P_*(\cH_k^{2n-1}).\] Here, the middle map is induced by
  the hyperbolic map and projecting to the base. All homology is taken with rational
  coefficients.
\end{proposition}

To make the above analysis as satisfying as possible, it is desirable to prove
the following algebraic conjecture:  

\begin{conjecture}
  The map $H_*(SO(2n;k);\Q) \to \tensor[^+]{H_*(SO(2n;k(i));\Q)}{}$ is an
  isomorphism. 
\end{conjecture}

If this conjecture were true then the zigzags in Corollary~\ref{cor:zigzags}
would be shortened to a length-$2$ because the rightmost inclusions would be
isomorphisms.

\section{Proof of Theorem~\ref{thm:reallycool}} \label{sec:twoproofs}

In this section we prove Theorem~\ref{thm:reallycool}.  First, some notation.
Write $\hat \I_d$ for the category whose objects are sequences
$\vec A = (b,a_1,\ldots,a_i)$ of nonnegative integers such that
$b+a_1+\cdots+a_i = d$ and all $a_i$ are positive, and morphisms defined as in
Definition~\ref{def:I}.  Recall the definition of $F^{\vec A}_\dotp$ in
Definition~\ref{def:flagspace}.  Let $\Phi^\eqvt: \hat\I_n \rto \Top_*$ be
defined by
\[\Phi^\eqvt(\vec A) \defeq S^\tw \smash F_\dotp^{\vec A}\]
and $\Phi:\hat\I_d \rto \Top_*$ be defined by
\[\Phi(\vec A) = \Phi^\eqvt(\vec A)_{hI(X)}.\] The functor $\Phi^\eqvt$ is defined the same as
the definition of $\YY$ in Definition~\ref{def:Dehnspace}, extended from $\I_d$ to
$\hat \I_d$.  Define
\[Z^\eqvt \defeq \tcofib \Phi^\eqvt \qqand Z \defeq \tcofib \Phi.\]   Then (as
total homotopy cofibers and homotopy coinvariants commute)
\[(Z^\eqvt)_{hI(X)} \simeq Z.\]

Surprisingly, it is possible to identify the homotopy type of $Z^\eqvt$.
\begin{proposition} \label{prop:Zid}
  There is an $I(X)$-equivariant weak equivalence 
  \[Z^\eqvt \simeq S^\tw \smash S^{d}.\]
  Here, the $I(X)$-action is trivial on the $S^{d}$-coordinate and acting by
  the determinant on $S^\tw$.
\end{proposition}
As the proof of this is technical we postpone it to the end of the section; for
now we assume it and complete the proof of Theorem~\ref{thm:reallycool}.  Note
that this proposition is an integral statement; it is not necessary to invert
$2$.

We now characterize $Z_{hI(X)}$ from a different perspective.
For $\vec A = (b,a_1,\ldots,a_i)$, by Lemma~\ref{lem:fratwe},
\[\left( S^\tw \smash F^{\vec A}_\dotp \right)_{hI(X)} \mathrel{\simeq_{[2]}} (S^\tw\smash
  F_\dotp^W)_{hI(X^b)} \smash \bigwedge_{j=1}^i (S^\tw \smash
  F_\dotp^{V_j})_{hO(a_j)} .\] By Proposition~\ref{prop:evencontr}, if any of
the $a_j$ are odd then
$(S^\tw \smash F_\dotp^{V_i})_{hO(a_i)} \mathrel{\simeq_{[2]}} *$; thus if
$\vec A$ has some $a_j$ odd then $\Phi(\vec A)$ is contractible.  For any atomic
morphism
$\iota: (b,a_1,\ldots,a_i) \to
(b,a_1,\ldots,a'_\ell,a''_\ell,a_{\ell+1},\ldots,a_i)$ (where
$a_\ell = a'_\ell + a''_\ell$) we say that the morphism is \textsl{in direction
  $r$} if
\[r = a''_\ell + a_{\ell+1} + \cdots + a_i.\] If $r$ is odd then
$\Phi(b,a_1,\ldots,a'_\ell,a''_\ell,\ldots,a_i)_{hI(X)}$ is contractible; thus
all morphisms in $\Phi(\hat\I_d)$ in odd directions have contractible codomain.
Note that $\I_d$ is exactly the subcategory of $\hat\I_d$ containing all atomic
morphisms in even directions.

Note that there are $\lfloor \frac{d-1}{2} \rfloor$ even directions and
$\lfloor \frac{d+1}{2} \rfloor$ odd directions.

It is possible to compute total homotopy cofibers \textsl{iteratively}: taking
all cofibers in a single direction $r$, it produces a cube one dimension lower;
the total homotopy cofiber of this cube is equivalent to the total homotopy
cofiber of the original cube.  Take homotopy cofibers in all of the even
directions first: this leaves a $\lfloor \frac{d+1}{2} \rfloor$-cube with a
single entry $\sY^X_{hI(X)}$ (at the source) and all other entries contractible;
since the homotopy cofiber of any map $X \to *$ is $\Sigma X$,
\[Z \simeq \Sigma^{\lfloor\frac{d+1}{2}\rfloor}(\sY^X_{hI(X)}).\]

By the homotopy orbit spectral sequence (see Proposition~\ref{prop:hoss}) and
Proposition~\ref{prop:Zid},
\[H_i \left((Z^\eqvt)_{hI(X)};\Zinv\right) \cong H_{i-(d+1)}\left(I(X);
    \Zinv^\tw\right).\] Thus
\[H_i\left(\sY_{hI(X)};\Zinv\right) \cong H_{\lfloor
    \frac{d+1}{2}\rfloor+i}\left(Z;\Zinv\right) \cong
  H_{\lfloor\frac{d+1}{2}\rfloor+i}\left((Z^\eqvt)_{hI(X)};\Zinv\right) \cong
  H_{i-\lfloor\frac{d-1}{2}\rfloor}\left(I(X);\Zinv^\tw\right),\] completing the proof of the theorem.\qed

It remains to prove Proposition~\ref{prop:Zid}.

We begin by computing the homotopy cofiber of a single Dehn invariant.  In order
to be able to do this for any general map in the cube, it is necessary to
generalize the definition of the Dehn invariant.

Let $W,U_1,\ldots,U_i$ be any decomposition of $X$ into orthogonal subspaces.
Define
\[d_j = \dim W + \sum_{\ell=1}^j \dim U_\ell \qquad j \geq 0.\] (Thus
$d_0 = \dim W$ and $d_i = \dim X$.)  Let $\ell$ be an integer distinct from
$d_1,\ldots,d_i$.  Let $j$ be the minimal index such that $d_j > \ell$.  For
convenience, define
\[D_\ell: F_\dotp^W \redjoin F_\dotp^{U_1} \redjoins F_\dotp^{U_i} \rto \bigvee_{\substack{V_j
      \subseteq U_j \\ \dim V_j = \ell-d_{j-1}}} F_\dotp^W \redjoin F_\dotp^{U_1} \redjoins F_\dotp^{V_j}
  \redjoin F_\dotp^{V_j^\perp\cap U_j} \redjoins F_\dotp^{U_i}\]
to be $1 \redjoins D_{\ell-d_{j-1}} \redjoins 1$.

\begin{definition}
  For any subset $I \subseteq \{1, \ldots,d\}$ let $N_IF_\dotp^X$ be the subspace of
  $F_\dotp^X$ containing no subspace with dimension contained in $I$.
\end{definition}

This definition gives a convenient way to identify the total homotopy cofiber of
a Dehn cube.

\begin{lemma} \label{lem:quottcofib} Let $X$ be a pointed simplicial set, and
  let $Y_1,\ldots,Y_n$ be subspaces of $X$.  Write $P(n)$ for the partial order
  of subsets of $\{1,\ldots,n\}$.  Define a functor
  \[F:P(n) \rto \Top_*  \qquad \hbox{by} \qquad I \rgoesto X\bigg/\bigcup_{i\in I}
    Y_i,\]
  with the induced morphisms given by the quotient maps.  Then
  \[\tcofib F \simeq \Sigma^n \bigcap_{i=1}^n Y_i.\]
\end{lemma}

\begin{proof}
  We prove this by induction on $n$.  When $n = 0$ the cube is trivial and the
  statement holds.  When $n =1$ the cube is $X \rto X/Y$, and the total homotopy
  cofiber is $\Sigma Y$, as desired.

  Now consider the general case.  The total homotopy cofiber can be computed
  iteratively \cite[Proposition 5.9.3]{munsonvolic} by first taking cofibers in
  the direction of ``adding $n$ to a set'': the morphisms in which each subset
  $J \in P(n-1)$ is mapped to $J \cup \{n\}$.  Taking the homotopy cofiber for
  each such $J$ produces the cube $G: P(n-1) \rto \Top_*$ given by
  \[J \rgoesto \Sigma Y_n \bigg/ \bigcup_{j\in J} \Sigma (Y_j\cap Y_n).\]
  This is an $n-1$-cube of the same type as in the
  proposition; by the induction hypothesis, 
  \[\tcofib G \simeq \Sigma^{n-1} \bigcap_{i=1}^{n-1} \Sigma (Y_i \cap Y_n).\]
  $\Sigma (Y_i \cap Y_n)$ sits inside $\Sigma X$ as
  $\Sigma Y_i \cap \Sigma Y_n$; then
  \[\Sigma^{n-1} \bigcap_{i=1}^{n-1} \Sigma(Y_i\cap Y_n) = \Sigma^{n-1} \bigcap_{i=1}^n (\Sigma Y_i) \cap
    (\Sigma Y_n) =\Sigma^{n-1} \bigcap_{i=1}^n \Sigma Y_i = \Sigma^n
    \bigcap_{i=1}^n Y_i,\]
  as desired.
\end{proof}

\begin{proposition} \label{prop:totcofib} Let $I \subseteq \{1,\ldots,d\}$.
  Consider the sub-$|I|$-cube formed by $D_i$ for $i\in I$ and containing the
  initial point.  This cube has total homotopy cofiber
  $\Sigma^{|I|} N_IF_\dotp^X$.
\end{proposition}

\begin{proof}
  For conciseness, write $D_I$ for the composition of the $D_i$ for $i\in I$.
  Since the Dehn cube commutes, the order of composition is irrelevant.  Let
  $I = \{i_0,\ldots,i_{j-1}\}$.  We claim that
  \[D_I: F_\dotp^X \rto \bigvee_{\substack{W\oplus^\perp U_1 \oplus^\perp \cdots
        \oplus^\perp U_j = X \\ \dim W = i_0 \\ \dim U_\ell =
        i_\ell-i_{\ell-1}-\cdots-i_0\ \ell < j}} F_\dotp^W \redjoin F_\dotp^{U_1} \redjoins
    F_\dotp^{U_j}\] is isomorphic to the map
  \[F_\dotp^X \rto F_\dotp^X/\bigcup_{i\in I}N_{\{i\}}F_\dotp^X\]
  via the isomorphism 
  \[\bigvee_{\substack{W\oplus^\perp U_1 \oplus^\perp \cdots
        \oplus^\perp U_j = X \\ \dim W = i_0 \\ \dim U_\ell = i_\ell\ \ell < j}}
    F_\dotp^W \redjoin F_\dotp^{U_1} \redjoins F_\dotp^{U_j} \rto
    F_\dotp^X/\bigcup_{i\in I} N_{\{i\}}F_\dotp^X\]
  taking an $\ell_j$-simplex 
  \[(U_0 \subsets U_{\ell_0}, U_{\ell_0+1} \subsets U_{\ell_1}, \ldots,
    U_{\ell_{j-1}+1}\subsets  U_{\ell_j})\]
  to $\ell_j$-simplex corresponding to the flag
  \[U_0 \subsets U_{\ell_0} \subset U_{\ell_0}\oplus U_{\ell_0+1} \subseteq
    U_{\ell_0} \oplus U_{\ell_0+2} \subsets U_{\ell_0}\oplus \cdots \oplus
    U_{\ell_j}.\] Every flag in the image contains subspaces of all dimensions
  contained in $I$, and any face map that removes one of them takes the simplex
  to the basepoint.  This map is bijective on simplices, showing that it is an
  isomorphism of simplicial sets.

  Since 
  \[N_IF_\dotp^X \cong \bigcap_{i\in I} N_{\{i\}}F_\dotp^X,\] the Dehn cube is
  isomorphic to a cube of the form in Lemma~\ref{lem:quottcofib}.  Applying the
  lemma we see that the total homotopy cofiber is
  \[\Sigma^{|I|}\bigcap_{i=1}^{n-1} N_{\{i\}} F_\dotp^X = \Sigma^{|I|}N_IF_\dotp^X,\]
  as desired.
\end{proof}

This can be used to prove Proposition~\ref{prop:Zid}:
\begin{proof}[Proof of Proposition~\ref{prop:Zid}]
  By Proposition~\ref{prop:totcofib},
  $Z \simeq S^\tw \smash \Sigma^{d}N_{\{0,1,\ldots,d-1\}}F_\dotp^X$.  However,
  $N_{\{0,1,\ldots,d-1\}}F_\dotp^X \cong S^0$, as it has exactly two simplices in each
  dimension: the basepoint and $X = \cdots = X$.  The functor $S^\sigma
  \smash\cdot$ commutes with taking homotopy cofibers; thus
  \[Z  \simeq S^\tw \smash S^{d}.\]
\end{proof}

To finish up this section we use this calculation to prove
Lemma~\ref{lem:edge-desc}.  First, a few definitions.  Denote by $\vec g$ a
tuple $(g_1,\ldots,g_j)$ of elements in $I(X)$; this tuple can be of any length
$j$.  For $0 \leq \ell \leq j$ and a coefficient $m\in \Zinv$, define the
notation
\[d_\ell(m \vec g) \defeq
  \begin{cases}
    m(\det g)(g_2,\ldots,g_j) \caseif \ell = 0\\
    m(g_1,\ldots,g_{\ell+1}g_\ell,\ldots,g_j) \caseif 1 \leq \ell < j \\
    m(g_1,\ldots,g_{j-1}) \caseif j = \ell.
  \end{cases}\]
Write, for $1 \leq a \leq b \leq j$,
\[\Pi_a^b \vec g\defeq g_b\cdots g_a \qqand \amalg_a^b \vec g \defeq
  g_a^{-1}\cdots g_b^{-1} = \left(\Pi_a^b \vec g\right)^{-1}.\]

The double complex $C_{**}$ is a homologically-graded double complex in which
for $j \geq 0$ and $0 \leq i < d$ the group $C_{ij}$ is generated by symbols
of the form
\[(g_1,\ldots,g_j)\{x_1|\cdots|x_i\},\]
where $g_1,\ldots,g_j\in I(X)$ and $x_1,\ldots,x_i\in X$.
When $j$ is clear from context we sometimes write this $\vec g
\{x_1|\cdots|x_i\}$. 
Define boundary maps $\partial^h:C_{ij} \to C_{(i-1)j}$ and $\partial^v:
C_{ij} \to C_{i(j-1)}$ by 
\begin{align*}
  \partial^h (\vec g\{x_0|\cdots|x_i\}) &\defeq
                                          \vec g\sum_{\ell=0}^{i}
                                          (-1)^\ell
                                          \{x_0|\cdots|\widehat
                                          x_\ell | \cdots | x_i\} \qquad \hbox{and}
  \\
  \partial^v(\vec g\{x_0|\cdots|x_i\}) &\defeq
                                                   (d_0\vec g)\{g_1x_0|\cdots|g_1x_i\}
                                                   +
                                                   \sum_{\ell=1}^{j-1}(-1)^\ell
                                                   (d_\ell\vec g)
                                                   \{x_0|\cdots|x_i\}.
\end{align*}

\begin{lemma} \label{lem:tech-comp} Let $\sigma \defeq \sum_i m_i \vec g^{(i)}$
  represent a cycle in $H_{d}(I(X);\Zinv^\tw)$ and fix any point $x\in X$.
  Inside the total complex of $C_{**}$ the cycle $\sum_i m_i \vec g^{(i)}\{\}$
  is homologous to
  \[(-1)^{d+1}\sum_i m_i (\det \Pi_1^{d} \vec g^{(i)})\sum_{\ell=0}^{d} (-1)^\ell
    \Big\{x \Big| g_{d}^{(i)}\cdot x \Big| \cdots\Big| \widehat{\Pi_{d-\ell+2}^{d} \vec g^{(i)}\cdot x} \Big| \cdots \Big|
     \Pi_2^{d} \vec g^{(i)}\cdot  x \Big|
     \Pi_1^{d} \vec g^{(i)}\cdot  x\Big\},\] (where the $\ell=0$ term removes
  the $g_1^{(i)}x$).
\end{lemma}

\begin{proof}
  Fix $x\in S^d$.  For any $\vec g = (g_1,\ldots,g_j)$, any
  $1 \leq \lambda \leq j\}$, and any point $y\in S^d$, and any $m \in \Zinv$, define
  \begin{align*}
    \Delta^\lambda(m\vec g,y) &\defeq m(g_1,\ldots,g_\lambda)\Big\{ y \Big| \amalg_1^j \vec
    g \cdot x
    \Big| \amalg_1^{j-1}\vec g \cdot x \Big| \cdots \Big| \amalg_1^\lambda \vec
                                g \cdot x\Big\}
    \in C_{(j-\lambda+2)\lambda} \\
    B^\lambda(m(a_1,\ldots,a_j)) &\defeq m(g_1,\ldots,g_{\lambda})\Big\{
                                \amalg_1^j \vec g\cdot x \Big|
                                   \amalg_1^{j-1}\vec g\cdot  x \Big|
                                \cdots \Big| \amalg_1^\lambda \vec g \cdot
                                   x\Big\} \in C_{(j-\lambda+1)\lambda}.
  \end{align*}
  For any $\vec g = (g_1,\ldots,g_j)$ we have
  \begin{align*}\partial^v \Delta^\lambda(\vec g,y)
  &= \Delta^{\lambda-1}(d_0\vec g, g_1y) + \sum_{\ell=1}^{\lambda-1}(-1)^\ell
  \Delta^{\lambda-1}(d_\ell\vec g, y)  + (-1)^\lambda (g_1,\ldots,g_{j-1})\Big\{ y \Big| \amalg_1^j \vec g
    \cdot x
  \Big| \cdots \Big| \amalg_1^\lambda \vec g \cdot  x\Big\}, \hbox{ and} \\
  \partial^h \Delta^{\lambda-1}(\vec g, y)
  &= B^{\lambda-1}(\vec g) + (-1)^{j+1}\sum_{\ell=\lambda}^{j} (-1)^{\ell} \Delta^{\lambda-1}(d_{\ell}\vec g, y)
   + (-1)^{j+ \lambda} (g_1,\ldots,g_{\lambda-1}) \Big\{ y \Big|
    \amalg_1^j \vec g \cdot x
   \Big| \cdots \Big| \amalg_1^\lambda \vec g \cdot x\Big\}.
  \end{align*}

  Now consider $\sigma$, so that in the formulas above $j = d$.  Define
  \[\alpha^\lambda \defeq \sum_i m_i \bigg(\Delta^\lambda(d_0\vec g^{(i)}, g_1^{(i)}x) +
    \sum_{\ell=1}^{d} (-1)^\ell \Delta^\lambda(d_\ell\vec g^{(i)}, x)\bigg)
    \in C_{(d+1-\lambda)\lambda}.
  \]
  The above  calculations (which will have $j = d-1$) imply that
  \[\partial^v\alpha^\lambda = \partial^h \alpha^{\lambda-1},\]
  using the fact that (since $\sigma$ is a cycle)
  \[\sum_i m_i\sum_{\ell=0}^j (-1)^\ell B^{\lambda-1} (d_\ell \vec g^{(i)}) =
    0.\]
  Thus, in the total complex,
  \[\partial \sum_{\ell=1}^{d} (-1)^\ell \alpha^\lambda = \partial^v \alpha^1
    + (-1)^d\partial^h \alpha^{d}.\]
  Plugging in the definitions produces that $\partial^h \alpha^{d} = \sigma$ and
  \begin{align*}
    \partial^v \alpha^1
    &= \sum_i m_i (\det g_1^{(i)})\sum_{\ell=0}^{d} (-1)^\ell \Big\{ g_1^{(i)}x
      \Big| \amalg_2^{d} \vec g^{(i)}\cdot x \Big| \cdots
      \Big|\widehat{\amalg_2^{d+1-\ell}\vec g^{(i)}\cdot x}\Big| \cdots\Big| \amalg_2^2
      \vec g^{(i)}x \Big| x\Big\}  \\ & \qquad
     + \sum_i m_i \sum_{\ell=0}^{d} (-1)^\ell \Delta^0(d_\ell \vec g^{(i)},x).
  \end{align*}
  (Here we abuse notation and declare that $\amalg_2^{d+1} \vec g^{(x)} = g_1$.)
  As $\sigma$ is a cycle, the second sum is $0$.  To complete the proof observe
  that for any $d+1$-tuple of points $(y_0,\ldots,y_{d})$, the class
  $\sum_{\ell=0}^{d} \{ y_0 | \cdots | \widehat y_\ell | \cdots | y_{d}\}$
  is a horizontal cycle, and is therefore homologous to
  \[\sum_{\ell=0}^{d} \{ g\cdot y_0 | \cdots | \widehat {g\cdot y_\ell} | \cdots |
    g\cdot y_{d}\}\]
  for any $g\in I(X)$.  Thus each term in the sum (over $i$) for
  $\partial^v\alpha^1$ is a cycle.  Acting on the $i$-th term by $\prod_2^{d}
  \vec g^{(i)}$ gives the desired expression up to permuting the first element
  to the end.  As this requires $d$ swaps, it changes the sign by $(-1)^d$.
\end{proof}

We are now ready to prove the general case of Lemma~\ref{lem:edge-desc}:

\begin{lemma} \label{lem:edge-desc-k}
  Let $d = 2n$ or $2n-1$.
  In the spectral sequence of (\ref{diag:ss}) the map
  \[\epsilon_{n-1}: H_{d}(I(X);\Zinv^\tw) \to H_{n-1}\P_*(X)\]
  is induced by the map taking a chain $(g_1,\ldots,g_{d})$ to the sum
  \[\bigg(\prod_{i=1}^{d} \det(g_i)\bigg)\sum_{\sigma\in
      \Aut\{0,\ldots,d\}}  \sgn(\sigma)    \left[V_{\sigma,0} \subseteq V_{\sigma,1} \subsets V_{\sigma,d}\right].\]
  Here we define
  \[V_{\sigma,i} =
    \rspan(h_{\sigma(0)}x_0,h_{\sigma(1)}x_0,\ldots,h_{\sigma(i)}x_0),\]
  with $h_{d} = 1$, $h_i = g_{d}\cdots g_{d-i}$, and $x_0$ any fixed
  point in $X$.

  When $k=\R$ this is the class of the $d$-simplex with vertices
  \[\{h_0x_0,h_1x_0,\ldots,h_{d}x_0\}.\]
\end{lemma}

\begin{proof}
  The map $\epsilon_{n-1}$ is induced by the edge homomorphism $\tcofib
  \YY_{hI(X)} \to \tcofib \hat\YY_{hI(X)}$, where
  \[\hat\YY(\vec A) \defeq
    \begin{cases}
      \YY((d)) \caseif \vec A = (d) \\
      * \caseotherwise.
    \end{cases}\] (This is the quotient of the $n-1$-st filtration level by the
  $n-2$-nd in the spectral sequence for the total homotopy cofiber of a cube.)
  Both of these total homotopy cofibers can be computed simultaneously and
  $I(X)$-equivariantly via the methods above.  As all maps in $\hat \YY$ map to
  the basepoint, each direction will simply suspend the $(d)$-case.  This
  implies that $\epsilon_{n-1}$ is the map on $H_{2d+1}$ induced by taking
  $I(X)$-invariants of the map
  \[S^\tw \smash \Sigma^{d}S^0 \to^{S^\tw \smash \iota_0} S^\tw \smash
    \Sigma^{d} F^X_\dotp\]
  where $\iota_0$ is the inclusion of the $0$-skeleton into $F^X_\dotp$.

  This map can be described in an alternate manner.  Model homotopy coinvariants
  via taking an extra simplicial direction (see, for example, \cite[Example
  IV.1.1]{goerssjardine}) and take the double chain complex associated to a
  bisimplicial set.  Then the homology of the geometric realization is
  isomorphic to the homology of this total complex.  Due to the suspension
  coordinates, the bottom $d+1$ rows of this double complex are $0$.  Above this
  (assuming that the group-coordinate is vertical and the flag-coordinate is
  vertical) the map above includes the standard bar construction for
  $H_*(I(X);\Zinv^\tw)$ as the leftmost column.  In order to show that the given
  formula for $\epsilon_{n-1}$ holds it is therefore sufficient to show the
  following: given a cycle $\sigma$ in $C_{d}(I(X);\Zinv^\tw)$ (which lies in
  coordinate $(0,2d+1)$ in the double complex), it is homologous to the cycle
  given by the formula in the statement of the lemma (which lies in coordinate
  $(d,d+1)$).
  
  In the double complex associated to the above bisimplicial construction, the
  group at coordinate $(m, \ell+d+1)$ is generated by diagrams of the form
  \begin{diagram-numbered}{diag:1}
    { X_0 & X_1 & \cdots & X_m \\
      g_1X_0 & g_1X_1 & \cdots & g_1X_m \\
      \vdots & \vdots & & \vdots \\
      g_\ell X_0 & g_\ell X_1 & \cdots & g_\ell X_m \\};
    \cofib{1-1}{1-2} \cofib{1-2}{1-3} \cofib{1-3}{1-4}
    \cofib{2-1}{2-2} \cofib{2-2}{2-3} \cofib{2-3}{2-4}
    \cofib{4-1}{4-2} \cofib{4-2}{4-3} \cofib{4-3}{4-4}
    \to{1-1}{2-1} \to{1-2}{2-2} \to{1-4}{2-4}
    \to{2-1}{3-1} \to{2-2}{3-2} \to{2-4}{3-4}
    \to{3-1}{4-1} \to{3-2}{4-2} \to{3-4}{4-4}
  \end{diagram-numbered}
  We denote such a diagram by
  \[(g_1,\ldots,g_\ell)[X_0 \subsets X_m].\] We define
  \[\{x_0|\cdots|x_i\} = \sum_{\sigma\in \Aut\{0,\ldots,i\}} \sgn(\sigma)
    [\rspan(x_{\sigma(0)}) \subseteq \rspan(x_{\sigma(0)},x_{\sigma(1)}) \subsets
    \rspan(x_{\sigma(0)},\ldots,x_{\sigma(i)}) \subseteq X].\] This double complex
  contains as a (vertically shifted by $d+1$) subcomplex the complex $C_{**}$ of
  Lemma~\ref{lem:tech-comp}; the conclusion of the lemma is exactly the desired
  formula.
  
  The last claim in the lemma follows by Theorem~\ref{thm:f<>p}.  
\end{proof}

\appendix

\section{Comparing RT-buildings to the classical
  constructions} \label{app:classical}

In this appendix we prove our claims in Theorem~\ref{thm:f<>p} and
Theorem~\ref{thm:RDehneq} that the construction of the scissors congruence
groups in our account agrees with the classical constructions.  To begin we
introduce an object closely related to the configuration space of points in $X$.
The simplices in this space are tuples of points in $X$; unlike in the
configuration space, points are allowed to be repeated, and this can produce
nondegenerate simplices.  For example, for any two distinct points $a,b\in X$,
$(a,b,a)$ is a nondegenerate $2$-simplex in $\Tpl^m_\dotp(X)$.

\begin{definition}
  $\Tpl^m_\dotp(X)$ is the simplicial set whose $i$-simplices are given by the
  subset of $\prod_{j=0}^i X$ of those tuples $(x_0,\ldots,x_i)$ such that any
  subset of the tuple has a nondegenerate span of dimension at most $m$.  The
  $j$-th face map is given by dropping the $j$-th element of the tuple; the
  $j$-th degeneracy is given by repeating the $j$-th element of the tuple.
\end{definition}

The homology of $\Tpl^m_\dotp(X)$ is directly related to scissors congruence
groups, as the following results illustrate:

\begin{theorem}{\cite[Theorem 2.10]{dupont}} \label{thm:toSC} Let $k = \R$.  The
  map taking a tuple of points to its convex hull defines a
  $I(\cH^n)$-equivariant isomorphism
  \[H_n(\Tpl^n_\dotp(\cH^n)/\Tpl^{n-1}_\dotp(\cH^n))^\tw \rto \P(\cH^n,1).\] Here,
  $\cdot^\tw$ means that the action is \emph{twisted} by the determinant: for any
  $g\in I(\cH^n)$, $g$ acts on a homology on the left by $(-1)^{\det g}$ as well
  as by the usual action on $\cH^n$.
\end{theorem}

The spherical case is more complicated.  Recall the map $\Sigma$, defined as the
suspension of a polytope, from Definition~\ref{def:spherangle}.

\begin{theorem}[{\cite[Corollary 5.18]{dupont_polytope}}] \label{thm:tospherSC}
  The map taking a simplex to its convex hull induces a $O(n+1)$-equivariant
  isomorphism
  \[H_n(\Tpl^n_\dotp(S^n)/\Tpl^{n-1}_\dotp(S^n))^\tw \rto
    (\coker \Sigma).\] In particular, since $\Sigma$ is
  $O(n+1)$-equivariant, $\Sigma$ induces an isomorphism on coinvariants 
  \[H_0(O(n+1), H_n(\Tpl^n_\dotp(S^n)/\Tpl^{n-1}_\dotp(S^n)^\tw) \rto
    \P(S^n,O(n+1)).\]
\end{theorem}

It follows that in order to relate scissors congruence groups and the homology
of $S^\tw\smash F_\dotp^X$ it suffices to show that $\Tpl_\dotp^{\dim
  X}(X)/\Tpl_\dotp^{\dim X-1}(X)$ and $F_\dotp^X$ are $I(X)$-equivariantly
homotopy equivalent. 

We begin with a basic lemma about the homotopy type of $\Tpl_\dotp(X)$ in the
absence of dimension restrictions:
\begin{lemma} \label{lem:Tpl*}
  \[\Tpl_\dotp^{\dim X}(X) \simeq *.\]
\end{lemma}

\begin{proof}
  By \cite[Proposition 2.2.1]{cathelineau_proj}, since $k$ is infinite
  $\tilde H_*(\Tpl_\dotp^{\dim X}(X)) = 0$.  (In fact, Cathelineau proves this
  only with rational coefficients, but his proof works equally well integrally.)
  To see that $\Tpl_\dotp^{\dim X}(X)$ is contractible it suffices to check that
  it is simply-connected.  By \cite[Proposition 2.2.2]{cathelineau_proj} for any
  pair of points $(x,y)$ spanning a subspace of $X$, the subset $W_{x,y}$ of
  those points in $X$ such that $(x,y,w)$ spans a subspace is a Zariski-open
  subspace of $X$.  Suppose that we are given a loop represented by the sequence
  of $1$-simplices $(x_0,x_1),(x_1,x_2),\ldots,(x_i,x_0)$.  Then, since $k$ is
  infinite, there exists a point $w$ such that $(x_j,w)$ spans a subspace for
  all $j$, and the loop is homotopic to a loop of the form $(x_0,w),(w,x_0)$.
  This is contracted by the $2$-simplex $(x_0,w,x_0)$, so
  $\Tpl_\dotp^{\dim X}(X)$ is contractible.
\end{proof}

For any simplicial set $K_\dotp$, let $\Sd K_\dotp$ be the barycentric
subdivision of $K_\dotp$ \cite[Section III.4]{goerssjardine}.  Define the map
$h: \Sd \Tpl^m_\dotp(X) \rto T^m_\dotp(X)$ to be the map induced by taking a
tuple of points in $X$ to their span.  More explicitly, an $i$-simplex in
$\Sd \Tpl^m_\dotp(X)$ is a sequence
$\vec x_0 \subseteq \vec x_1 \subsets \vec x_i$, where $\vec x_j$ is a tuple in
$X$ and $\vec x_{j-1}$ is an (ordered) subset of $\vec x_j$ for all $j$.  Taking
the spans of each tuple produces an $i$-simplex in $T^m_\dotp(X)$; as taking
spans is $G$-equivariant, this map is $G$-equivariant.

\begin{proposition} \label{prop:tup->t}
  The map
  \[h:\Sd \Conf^m_\dotp(X) \rto T^{m}_\dotp(X)\] induced by taking tuples in $X$
  to their spans is a $G$-equivariant weak equivalence.
\end{proposition}
\begin{proof}
  We use Theorem A' \cite[p.578]{gillet_grayson}, which states that a map of
  simplicial sets is a weak equivalence if the ``naive left homotopy fiber'' above
  every simplex in the codomain is contractible.  Here, for a given a
  $q$-simplex $y\in T^{m}_{q} (X)$ represented by
  $(U_0 \subseteq \cdots \subseteq U_q)$, the naive left homotopy fiber is the
  simplicial set
  \[(h|y)_p = \Big\{(\vec x_0 \subsets \vec x_p) \in \Sd \Conf^m_{p}
    (X)\,\Big|\,  \hbox{all entries of $\vec x_p$ are in $U_0$} \Big\}.
  \]
  (For a precise definition of the naive homotopy fiber see for example
  \cite[Defn.~3.1]{grayson_additivity}.)  In this case, $(h|y)_p$ is isomorphic
  to the simplicial set $\Sd \Conf^m_{\dotp} (U_0)$; as this is isomorphic to
  $\Sd \Conf^{\dim U_0}_\dotp(U_0)$ it is contractible by Lemma~\ref{lem:Tpl*}.
\end{proof}

The $m = \dim X$ case of the following theorem shows that $F^X_\dotp$ is
$G$-equivariantly homotopy equivalent to a quotient of tuple spaces.

\begin{theorem} \label{thm:tpl<>flag} For all $m \geq 0$ 
  \[\Tpl_\dotp^m(X)/\Tpl_\dotp^{m-1}(X) \simeq
    T^m_\dotp(X)/T^{m-1}_\dotp(X)\]
  via a zigzag of $G$-equivariant maps. 
\end{theorem}


\begin{proof}
We have the $G$-equivariant
commutative diagram
\begin{diagram}
  { \Conf_\dotp^{m-1}(X) & \Sd \Conf_\dotp^{m-1}(X) & T_\dotp^{m-1}(X) \\
    \Conf_\dotp^m(X) & \Sd \Conf_\dotp^m(X) & T_\dotp^m(X) \\};
  \we{1-2}{1-3}_h \we{2-2}{2-3}_h
  \we{1-2}{1-1} \we{2-2}{2-1}
  \cofib{1-2}{2-2} \cofib{1-3}{2-3} \cofib{1-1}{2-1}
\end{diagram}
where the vertical maps are injective on all $i$-simplices, hence cofibrations.
Taking vertical cofibers gives the desired result, as the cofibers of the
vertical maps are also the homotopy cofibers.
\end{proof}

\begin{corollary} \label{cor:concentrated}
  \[\tilde H_i(F_\dotp^X) = 0 \qquad \hbox{for }i \neq \dim X.\]
\end{corollary}

\begin{proof}
  By definition $F_\dotp^X = T^{\dim X}_\dotp(X)/ T^{\dim X-1}_\dotp(X)$.  As
  all simplices of $T^{\dim X}_\dotp(X)$ above $\dim X$ are degenerate (since
  they must repeat at least one subspace) it must be the case that $\tilde
  H_i(F_\dotp^X) = 0$ for $i > \dim X$.  By Theorem~\ref{thm:tpl<>flag},
  $F_\dotp^X \simeq \Tpl^{\dim X}_\dotp(X)/\Tpl^{\dim X-1}_\dotp(X)$.  However,
  all simplices of $\Tpl^{\dim X}_\dotp(X)$ of dimension less than $\dim X$ are
  contained in $\Tpl_\dotp^{\dim X-1}(X)$, since the span of $i$ points has
  dimension at most $i-1$.  Thus $\tilde H_i(F_\dotp^X) = 0$ for $i < \dim X$.  
\end{proof}

Using these results we can finally prove Theorem~\ref{thm:f<>p}:
\begin{proof}[Proof of Theorem~\ref{thm:f<>p}]
  Theorems~\ref{thm:toSC} and \ref{thm:tospherSC}, together with
  (\ref{eq:SCtoH0}) demonstrate that scissors congruence groups are group
  homology with coefficients in $H_n(F_\dotp^X)^\tw$; this is exactly
  $H_{n+1}(S^\tw \smash  F_\dotp^X)$.  By the homotopy orbit spectral sequence
  (Proposition~\ref{prop:hoss}) this is $H_{n+1}((S^\tw \smash F_\dotp^X)_{hI(X)})$, as
  desired.  The formula for the class represented by the vertices of a simplex
  follows from Theorem~\ref{thm:tpl<>flag} and the fact that on homology the
  inverse to the map $\Sd\Tpl_\dotp^{m-1}(X) \rwe \Tpl_\dotp^{m-1}(X)$ is
  given by the formula
  \[[x_0,\ldots,x_n] \rgoesto \sum_{\sigma\in \Aut\{1,\ldots,n\}} \sgn(\sigma)
    \big[(x_{\sigma(0)})' \subseteq (x_{\sigma(0)},x_{\sigma(1)})' \subsets
    (x_{\sigma(0)},\ldots,x_{\sigma(n)})'\big].\]
  Here, $(x_{\sigma(0)},\ldots,x_{\sigma(i)})'$ is ordered \textsl{not} by the
  ordering $0,\ldots,i$ but rather by the ordering induced on the $x$'s from the
  tuple $(x_0,\ldots,x_n)$.
\end{proof}

We wrap up this section by proving our claim in Theorem~\ref{thm:RDehneq}, that
the derived definition of the Dehn invariant is compatible with the classical
Dehn invariant.  As we saw in Theorem~\ref{thm:f<>p}, in order to translate
between classical scissors congruence groups and RT-buildings we must take
``semi-coinvariants,'' and thus the twist by $S^\tw$
(Definition~\ref{def:Ssigma}) appears here as well. 

\begin{proof}[Proof of Theorem~\ref{thm:RDehneq}]
  Rewriting $\hat D_i$ using Lemma~\ref{lem:componentDehn}, we see that it
  suffices to construct an $I(X)$-equivariant diagram relating
  $\bigoplus_U \hat D_U$ to $H_{n+1}(S^\tw \smash D_i)$.
  
  For a geometry $W$ of dimension $i$, write
  \[R_\dotp(W) \defeq \Sd \Tpl_\dotp^i(W)/\Sd \Tpl_\dotp^{i-1}(W)\] for the
  quotient of barycentric subdivisions $\Sd$. To define
  $D^R_i: R_\dotp(X) \rto \bigvee_U R_\dotp(U) \redjoin R_\dotp(U^\perp)$
  consider a $j$-simplex of $R_\dotp(W)$: this is represented by a sequence
  $T_0 \subsets T_j$ of tuples of points in $W$ such that the span of $T_j$ is
  $W$.  If there exists a maximal $\ell$ such that
  $\dim \operatorname{span} T_\ell = i$, we map this $j$-simplex to the simplex
  $(T_0 \subsets T_\ell)\smash (\pr_{U^\perp}T_{\ell+1} \subsets \pr_{U^\perp}
  T_j)$, indexed by $\operatorname{span} T_\ell$.  Otherwise, we map to the
  basepoint.  This is a well-defined simplicial map for the same reason that
  $D_i$ is.

  Consider the following diagram:
  \begin{diagram}[7em]
    { H_{n+1} (S^\tw \smash F_\dotp^X) & H_{n+1}\bigg(S^\tw \smash
      \bigvee_{\substack{U \subseteq X
          \\ \dim U = i}} F_\dotp^U \redjoin F_\dotp^{U^\perp}\bigg) \\
      H_{n+1}(S^\tw \smash R_\dotp(X)) &
      H_{n+1}\bigg( S^\tw \smash \bigvee_{\substack{U \subseteq X \\ \dim U =
          i}} R_\dotp (U) \redjoin R_\dotp(U^\perp)\bigg)
      \\
      \P(X,1) & \bigoplus_{\substack{U \subseteq X \\ \dim U = i}} \P(U,1)
      \otimes \P(S^{n-i-1},1) \\
    };
    \to{1-1}{1-2}^{H_{n+1}(S^\tw \smash D_i)}
    \we{2-1}{1-1}_h \we{2-2}{1-2}_h
    \to{2-1}{2-2}^{H_{n+1}(S^\tw \smash D^R_i)}
    \to{3-1}{3-2}^{D_\Q}
    \we{2-1}{3-1}_p \we{2-2}{3-2}_p
  \end{diagram}
  Here, the vertical maps $h$ are induced by the map $h$ in
  Theorem~\ref{thm:tpl<>flag} (and are thus isomorphisms).  The vertical maps
  $p$ are defined as in Theorem~\ref{thm:f<>p} (with $G$ trivial) and are
  therefore isomorphisms.  Since all maps in this diagram are
  $I(X)$-equivariant, the lemma follows.
\end{proof}

\section{Technical miscellany}
\label{app:hocofib}

\subsection{Reduced joins} \label{sec:redjoin}

In this section we restate the definition of a reduced join and prove several
important properties.

\begin{definition}
  We define $X \redjoin Y$ to be the simplicial set with
  \[(X \redjoin Y)_n = \bigvee_{i=0}^{n-1} X_i \smash Y_{n-i-1}.\]
  On a simplex $(x,y)\in X_i \smash Y_{n-i-1}$, the map $d_j$ is defined to be
  $d_j\smash 1:X_i \smash Y_{n-i-1}\rto X_{i-1} \smash Y_{n-i-1}$ if $j \leq i$
  and $1 \smash d_{j-i-1}:X_i \smash Y_{n-i-1} \rto X_i \smash Y_{n-i-2}$ if $j
  \geq i+1$.  The degeneracies are defined similarly.  
\end{definition}

\begin{lemma} \label{lem:joindist}
  Reduced joins distribute over wedge products.
\end{lemma}

\begin{proof}
  We have
  \[\left(\bigvee_{\alpha\in A} (X_\alpha \redjoin Y)\right)_n =
    \bigvee_{\alpha\in A} \bigvee_{i+j=n-1} (X_\alpha)_i \smash Y_j =
    \bigvee_{i+j=n-1} \left(\bigvee_{\alpha\in A} X_\alpha\right)_i \smash Y_j =
    \left(\left(\bigvee_{\alpha\in A} X_\alpha\right)\redjoin Y\right)_n.\]
  Since each step of this expression commutes with simplicial maps, the two are
  isomorphic as simplicial sets.  
\end{proof}

\begin{lemma} \label{lem:joincommute}
  Let $f:X \rto Y$ be a quotient of simplicial sets.  Then the map
  $f\redjoin 1:X \redjoin Z \rto Y \redjoin Z$ is also a quotient of simplicial
  sets.  $(f\redjoin 1)^{-1}(*) = f^{-1}(*) \redjoin Z$.
\end{lemma}

\begin{proof}
  It suffices to show that every nonbasepoint simplex in the codomain has a
  unique preimage in the domain.  Consider a non-basepoint $n$-simplex in
  $Y \redjoin Z$; this is a pair of the form $(y_i,z_j)$ with $y_i\in Y_i$,
  $z_j\in Z_j$ and $i+j=n-1$.  As $y_i \in Y_i$ is non-basepoint, it has a
  unique preimage $x_i\in X_i$.  As the given map takes $(x,z)$ to $(f(x),z)$
  the preimage of $(y_i,z_j)$ is exactly $(f^{-1}(y_i),z_j)$, which is unique.

  The simplices that map to the basepoint are exactly those that $f$ maps to the
  basepoint, with anything in the $Z$-coordinate.
\end{proof}

We end by giving a map relating the smash product and the reduced join.

\begin{lemma} \label{lem:smashToJoin} Let $X$ and $Y$ be pointed simplicial
  sets.  The map $f:S^1 \smash X \smash Y \rto X \redjoin Y$ given by sending
  $(i,x,y)\in (S^1\smash X \smash Y)_n$ to $(d_i^{n-i+1}x, d_0^{i+1}y)$ is a
  simplicial weak equivalence.
\end{lemma}

\begin{proof}
  The fact that $f$ is well-defined is direct from the definition.
  We define $X\mathrel{*_w} Y$ to be the double mapping cylinder of the diagram
  \[X \lto^{\pr_X} X\times Y \rto^{\pr_Y} Y.\] We can thus think of
  $X \mathrel{*_w} Y$ as the quotient of $I\times X \times Y$ given by the
  mapping cylinder relations $(x,0,y) \sim (x',0,y)$ and
  $(x,1,y) \sim (x, 1, y')$ for all $x,x'\in X$ and $y,y'\in Y$.  Consider the
  following commutative square:
  \begin{diagram}
    { X \mathrel{*_w} Y & S^1 \smash X \smash Y \\
      X * Y & X \redjoin Y \\};
    \arrowsquare{g}{f'}{f}{g'}
  \end{diagram}
  The maps $g$ and $g'$ are both weak equivalences because they are quotients by
  contractible subspaces.  The map $f'$ is a weak equivalence by \cite[Corollary
  3.4]{fritschgolasinski}.  Thus, by 2-of-3, $f$ is a weak equivalence, as desired.  
\end{proof}

\subsection{Homotopy coinvariants}

All of the results in this section are well-known to experts, although we could
not find references for them for the specific cases we were interested in.

\begin{definition}[{\cite[Example IV.1.10]{goerssjardine}}]
  Let $X$ be a (pointed) simplicial set with an action by a discrete group $G$.
  The \emph{homotopy coinvariants} (or \emph{homotopy orbits}) of $G$ acting on
  $X$, denoted $X_{hG}$, is the diagonal of the bisimplicial set with
  $(m,n)$-simplices given by diagrams
  \[x \to^{g_1} g_1x \to^{g_2} g_2g_1x \to \cdots \to^{g_n} g_n\cdots g_1x\] for
  $x\in X_m$.
\end{definition}

Directly from the definition we see that $*_{hG} \cong *$ and
$S^0_{hG} \cong BG_+$.

\begin{remark}
  This agrees with the more standard definition of homotopy coinvariants,
  defined as
  \[X_{hG} \defeq EG_+\smash_G X.\]
  (In the unpointed context, $\smash$ is replaced by $\times$.)
\end{remark}

There is a spectral sequence for computing the homology of the homotopy orbits
from the group homology of $G$ with coefficients in the homology of $X$:

\begin{proposition} \label{prop:hoss} There is a spectral sequence
  \[H_p(G, \tilde H_q(X)) \Rto \tilde H_{p+q}(X_{hG}).\]
\end{proposition}

The proposition holds for all simplicial sets with $G$-action, which is the case
of concern in this paper.

\begin{proof}
  Consider $X$ as an unpointed simplicial set; write this space $\bar X$.  The
  homology of the diagonal simplicial set of a bisimplicial set is the homology
  of the total complex of the associated simplicial abelian group.  The spectral
  sequence associated to a simplicial abelian group $A_{\dotp\dotp}$ has
  \[E^2_{p,q} = H^{\mathrm{vert}}_q H^{\mathrm{horiz}}_p (A_{\dotp\dotp}) \Rto
    H_{p+q}(\diag A_{\dotp\dotp}).\] Applying this in the current case to both
  $\bar X_{hG}$ and $*_{hG}$ gives us the following pair of spectral sequences:
  \[H_p(G, H_q(X)) \Rto H_{p+q}(\bar X_{hG}) \qqand H_p(G, H_q(*)) \Rto
    H_{p+q}(BG).\]
  The second is a retract of the first; if we take the other summand, we get a
  spectral sequence
  \[H_p(G, \tilde H_q(X)) \Rto \tilde H_{p+q}(X_{hG}),\]
  as desired.
\end{proof}




Lastly we present a technical proposition relating certain kinds of homotopy orbits.

\begin{proposition} \label{prop:restrict_coinv}
  Let $G$ be a group acting on a pointed simplicial set $X_\dotp$.  Suppose that
  $Y_\dotp$ is a subspace of $X_\dotp$ such that the following two conditions
  hold: 
  \begin{enumerate}
  \item If $g\in G$ is such that there exists a (non-basepoint) simplex
    $y\in Y_\dotp$ such that $g\cdot y\in Y_\dotp$ then for all $y'\in Y_\dotp$,
    $g\cdot y'\in Y_\dotp$.
  \item For all $n$ and for all $x\in X_n$ there exists $g\in G$ such that $g
    \cdot x \in Y_n$.
  \end{enumerate}
  Let $H$ be the subgroup of $G$ that takes $Y_\dotp$ to $Y_\dotp$.  Then
  \[(X_\dotp)_{hG} \simeq (Y_\dotp)_{hH}.\]
\end{proposition}

\begin{proof}
  Let $Z_{\dotp\dotp}$ be the bisimplicial set whose $(n,m)$-simplex consist of
  diagrams
  \[x_0 \rto^{g_1} x_1 \rto^{g_2} \cdots \rto^{g_m} x_m,\] where the
  $x_i\in X_n$ for $i = 0,\ldots,m$ and $g_i \cdot x_{i-1} = x_i$.  Then
  $\diag Z_{\dotp\dotp} = (X_\dotp)_{hG}$.  In addition, if we let
  $W_{\dotp\dotp}$ be the sub-bisimplicial set containing those diagrams where
  the $x_i\in Y_\dotp$ and the $g_i\in H$ then $\diag W_{\dotp\dotp} = Y_{hH}$.
  Thus it suffices to check that the inclusion
  $W_{\dotp\dotp} \rto Z_{\dotp\dotp}$ induces an equivalence on diagonals.  To
  prove this, it suffices (by \cite[Proposition IV.1.9]{goerssjardine}) to show
  that for all $n$, $W_{n\dotp} \rto Z_{n\dotp}$ is a weak equivalence of
  simplicial sets.

  $Z_{n\dotp}$ (resp. $W_{n\dotp}$) is the nerve of the category whose objects
  are $X_n$ (resp. $Y_n$) and whose morphisms are induced by the action of $G$
  (resp. $H$); call these categories $\C$ and $\D$.  $\D$ is clearly a
  subcategory of $\C$; thus to show that the map induces an equivalence on
  nerves it suffices to check that the inclusion is full and essentially
  surjective.  That it is full follows from condition (1), since since if we are
  given $y,y'\in Y_n$ then any $g$ such that $g \cdot y = y'$ is in $H$.  That
  it is essentially surjective follows from condition (2), since every element
  of $X_n$ is isomorphic via the action of $G$ to an element of $Y_n$.
\end{proof}

\subsection{The spectral sequence for the total homotopy cofiber of a cube}\label{subsec:thocofib}. 

The technical result that we need in order to understand the Dehn cube is the
spectral sequence for the total homotopy cofiber of a cube.  As the usual
spectral sequence is stated only for ordinary, rather than reduced, homology, we
state our analog here.  We use the notation introduced in Section~\ref{sec:goncharov}.

\begin{proposition} \label{prop:sstcofib}
  Let $F: \hat\I_n \rto \Top_*$ be a functor.  There is a spectral sequence
  \[\bigoplus_{\vec A = (b,a_1,\ldots,a_{n-p-1})} \tilde H_q(F(\vec A)) \Rto
    \tilde H_{p+q}(\tcofib F).\]
\end{proposition}

\begin{proof}
  By \cite[Proposition 9.6.14]{munsonvolic}, for a functor $G: \hat\I_n \rto \Top$
  there is a spectral sequence 
  \[\bigoplus_{\vec A = (b,a_1,\ldots,a_{n-p-1})} H_q(G(\vec A)) \Rto
    H_{p+q}(\tcofib G).\]
  Each of the spaces we have is pointed, thus the functor $C: \hat\I_n \rto \Top$
  defined by $C(\vec A) = *$ is a retract of $G$.  In particular, this means
  that the spectral sequence given by the kernel of the induced map $G \Rto C$
  is also a spectral sequence, which converges to $\ker (H_{p+q}(\tcofib G) \rto
  H_{p+q}(\tcofib C))$.  Since $\tcofib C \simeq *$, this reduces to the desired
  spectral sequence.
\end{proof}


\bibliographystyle{amsalpha}
\bibliography{CZ}

\end{document}